\documentclass{{amsart}}
\usepackage{amsmath,amssymb,amsthm,mathtools}
\usepackage{hyperref}
\usepackage{xcolor}

\newcommand{\Oval}{\mathcal{O}}

\newcommand{\Mval}{\mathfrak{M}}

\newtheorem{lema}{Lemma}[subsection]
\newtheorem{prop}[lema]{Proposition}

\theoremstyle{definition}
\newtheorem{sit}[lema]{Situation}
\newtheorem{defi}[lema]{Definition}

\title[One-dimensional commutative groups in ACVF and in PL$_0$]{One-dimensional commutative 
groups definable in algebraically closed valued
fields and in pseudo-local fields}

\author{Juan Pablo Acosta L\'opez}
\email{jupaaclo1393@gmail.com}

\author{Martin Hils}
\address{Martin Hils, Institut f\"ur Mathematische Logik und Grundlagenforschung, Universit\"at M\"unster, Einsteinstr. 62, D-48149 M\" unster, Germany}
\email{hils@uni-muenster.de}
\thanks{JPAL was was supported by the Alexander von Humboldt Foundation through a
Humboldt Research Fellowship for Postdoctoral Researchers. Both authors were partially supported by the German Research Foundation (DFG) via HI 2004/1-1 (part of the French-German ANR-DFG project GeoMod) and under Germany's Excellence Strategy EXC 2044-390685587, `Mathematics M\"unster: Dynamics-Geometry-Structure'.}


\begin{document}
\maketitle
\begin{abstract}
We give a complete list of the commutative one-dimensional groups definable in algebraically
closed valued fields and in pseudo-local fields, up to a finite index subgroup
and quotienting by a finite subgroup.
\end{abstract}
\section{Introduction}
In \cite{one-dimensional-p-adic} a classification of one-dimensional groups
definable in the $p$-adic numbers is established. Here we obtain a similar classification for definable commutative
groups of dimension $1$ over other valued fields.
Of special importance are the algebraically closed valued fields. 
This is done in Proposition \ref{one-dimensional-acvf} except in 
residue characteristic $2$ and $3$. The only part where this restriction appears
is in the study of elliptic curves, it would be good to remove it.

Another interesting field is the one obtained by letting $p$ vary in a $p$-adic field.
More precisely if we have $\mathcal{U}$ a non-principal ultra-filter on the set of prime numbers,
then $\Pi_{p}\mathbb{Q}_p/\mathcal{U}$ is elementary equivalent to $F((t))$
where $F$ is a pseudo-finite field of characteristic $0$ (we can also replace $\mathbb{Q}_p$
with any finite extension of it). 
The classification in this case is obtained in Proposition \ref{one-dimensional-psfv0}.

The strategy of the proof is similar to the one in \cite{one-dimensional-p-adic}.
The main new technical issue in the classification for these valued fields
is the presence of an infinite residue field.

To recall how this proof goes,
one uses the algebraization result for groups 
in a very general context obtained in \cite{mop},
see Section \ref{mopsec}, in order to obtain a one-dimensional 
algebraic group containing as a type-definable subgroup a ``large'' 
type-definable subgroup of the given definable group, after quotienting
out a finite subgroup. 
Here large means that the index is of cardinality smaller than the degree
of saturation of the model.

An important subproblem now becomes describing all type-definable subgroups
of a one-dimensional algebraic group. 

One-dimensional connected
algebraic groups over a perfect field have a classical classification:
they are either the additive group, the multiplicative group, a twisted multiplicative group or an elliptic curve.
So one has to describe the type-definable subgroups in each of these cases.
After this one uses logic 
compactness to go from the given type-definable group morphism to a 
definable local group morphism and from there to a suitably defined ind-definable group
morphism, which gives the classification. We try to reference the arguments
in \cite{one-dimensional-p-adic} when relevant so that we can focus 
on the modifications needed in this article. 

We mention that we give here an additional argument that allows us to 
get rid of the finite kernel in the classification in several cases,
for instance the argument shows this finite kernel is unnecessary in the 
previously obtained classification for the $p$-adic number case, see
Lemma \ref{finite-kernel}. We also mention that the hypothesis
of commutativity is now known to be unnecessary in all cases 
except perhaps the case of algebraically closed valued fields of
positive equicharacteristic, by work of Hasson and the first author (see \cite[Section 8]{acosta-hasson}), where
it is proven that a definable one-dimensional group contains
a finite index commutative subgroup.
The proof in the pseudo-local case requires the classification
obtained here.

The description of type-definable subgroups of algebraic one-dimensional groups proceeds
more or less as follows. For the given algebraic group $G$,
 often there are short exact sequences
$0\to G_0\to G\to \overline{G}$, and $0\to G_0^-\to G_0\to H$, where $H$ is an algebraic group over the residue field, and where $G_0$ and $G_0^-$ are suitable definable subgroups of $G$.
These maps sometimes satisfy that all definable subsets of the domain
 contain almost all fibers of the image, or at some points some more complicated
weakening of this property. This property appears often enough
that we call it opaqueness --- this term is borrowed from \cite{HasHruMac-Book}. Using opaqueness, one sees that the 
classification reduces
to a classification of definable subgroups of $G_0^-$,  $H$ and of $\overline{G}$, 
see Lemma \ref{opa-group}. 
For the kernel
$G_0^-$ one has a group filtration  
$w:G_0^-\to \Gamma$ with intermediate quotients $(k,+)$,
and in this case a proof analogous to the case of $\Oval$ 
with the filtration $v:\Oval\to \Gamma$ works.

The fact that the residue field is infinite appears in several places in
the described analysis. First the algebraic group $H$ is infinite,
which led us to the notion of opaqueness.
Second, $G_0^-$ has a subgroup which is analytically isomorphic to
$(\Oval,+)$ via a uniformization map. 
In the $p$-adic number case this subgroup is of finite index, so no
additional analysis is necessary. In the case 
of ACVF of positive residue characteristic, this subgroup is merely one
of the subgroups of $G_0^-$ corresponding to the filtration mentioned
above, which necessitates an additional argument for the case $G_0^-$.

As an aside, we mention that the
detailed information about the subgroups of algebraic groups that
we obtain, allows us
for instance to identify the maximal stably dominated subgroup in the case of 
an elliptic curve as defined in \cite{metastable}. See Proposition \ref{st-dom-ell}.

The methods here work for other valued fields when there is a good description of 
the type-definable subgroups of one-dimensional algebraic groups in the residue field,
and of the definable sets in the value group. So for example the proofs in 
Proposition \ref{one-dimensional-psfv0} adapt without change to $\mathbb{C}((t))$ and 
$F((t^\mathbb{Q}))$ for $F$ pseudo-finite of characteristic $0$.
As the information needed for the definable subsets of the value groups
is very detailed, we refrain from attempting a general framework, we believe
the cases treated are the ones of the most general interest.

In the last section we show that in the case of algebraically
closed valued fields of equicharacteristic $0$, the list given in the classification
has only the expected redundancies.

\subsection{Overview of the paper}
In Section \ref{prelimsec} we review the various preliminary results
used in this paper.

In Section \ref{notsec} we set some basic notation concerning valued
fields, model theory, and some specific notation about groups in valued
fields.

In Section \ref{valsec} we review the results on the model theory
of valued fields used in the paper, including quantifier elimination results
and a description of definable subsets of $K$.

Section \ref{rvsec} contains a result of quantifier elimination in
a short exact sequence of abelian groups, which is used to describe
the definable subsets of RV.

As mentioned in the introduction, 
Section \ref{mopsec} contains the algebraization result that allows
us to relate a definable group to an algebraic group.

Section \ref{anasec} contains a description of the model theory of 
valued fields with
 certain analytic functions added.
This is usually written in some general framework, so for convenience
of the reader, we detail
those particular cases which we will be making use of.

Sections \ref{ellsec} and \ref{tate-section} 
contain material on elliptic curves ---
we use \cite{silverman-1} and \cite{silverman-2} as general references. 
In algebraic contexts it is very common to restrict oneself to 
valued fields of rank 1, or even complete discrete valuation fields.
We work in $\omega$-saturated valued fields,
which are not of rank 1,
 so we hope these sections serve as a useful reference
on statements about elliptic curves used in this paper.

Section \ref{alggrosec} contains classical material about algebraic groups
and some basic facts on the model theory of algebraically closed fields.

Section \ref{opasec} deals with one of the central technical tools of the paper,
the notion of opaqueness mentioned in the introduction. Here
we show that various maps in valued fields are opaque. Several 
of the statements here are proven in the generality of an
arbitrary 
residue characteristic $0$ valued field. We also show
that the description of the subgroups of a definable group reduces
to a description of subgroups of the kernel and the quotient when the
quotient map is opaque. This concept allows us to treat more general
valued fields, and also streamlines some of the 
previous arguments on the classification in the $p$-adic number case.

Section \ref{ortsec} contains statements on the triviality of maps
from the value group and residue field into the main sort of a valued
field. 

Section \ref{subgrosec} contains the description of the type-definable
subgroups of one-dimensional algebraic groups in the valued fields
discussed in this paper.
Sections \ref{addgrosec}, \ref{mulgrosec}, \ref{ellgrosec} and 
\ref{twimulgrosec}
deal with the additive, multiplicative, elliptic and twisted multiplicative
cases, respectively. As mentioned in the introduction the proofs have 
some features in common, the use of opaqueness is one of them,
and arguments dealing with a certain group equipped with a filtration
is another. For this last one, we extract the common argument
to avoid repetition in Lemma \ref{pll} which is adapted to 
residue characteristic $0$ valued fields, and in Lemma \ref{acvfl}
which is adapted to ACVF. 
As an aside, note that for the case of ACVF with residue 
characteristic $0$ both arguments are valid.

Section \ref{maisec} contains our main classification results for one-dimensional definable groups in the valued fields considered.
The proofs are the same as in \cite{one-dimensional-p-adic} 
once the material above has 
been developed, and so, they are omitted.
As mentioned in the introduction we give here an argument that allows us
to remove the finite kernel in the classification in some cases,
including the previously obtained $p$-adic classification.

In Section \ref{isogrosec} we work with the case of algebraically closed valued
fields of residue characteristic $0$. In this context we show that the list
obtained in the classification is not redundant, in the sense that only the
expected groups are definably isomorphic. The main idea is to take
the Zariski closure of a definable group morphism to obtain an
algebraic morphism (see Lemma \ref{def-to-alg-mor}). 
We leave untreated the positive
residue characteristic case and the pseudo-local case.

\subsection*{Acknowledgement} We thank the anonymous referee for useful comments on an earlier version of this paper which helped improve the presentation.

\section{Preliminaries}\label{prelimsec}
\subsection{Notation}\label{notsec}
Throughout the paper for a valued field $K$ we use the notation $\Oval$ for
the valuation ring, $\Gamma$ for the value group and $\Mval$ for the maximal
ideal of $\Oval$. We denote $\Gamma_{\infty}$ the group $\Gamma$ together with one additional distinguished element denoted $\infty$. The valuation 
$K\to \Gamma_{\infty}$ is denoted by $v$.
 We denote $RV=K^{\times}/(1+\Mval)$ and $rv:K^{\times}\to RV$ the canonical
projection. We also denote $RV_0$ to be $RV$ together with an additional
element denoted $0$, 
and also denote by $rv$ the map $rv:K\to RV_0$ 
that sends $0$ to $0$.
If $r\in \Gamma$ we denote $B_r=\{x\in K\mid v(x)\geq r\}$ and
$B_r^-=\{x\in K\mid v(x)>r\}$. If $r>0$ denote
$U_r=\{x\in K\mid v(1-x)\geq r\}$ and $U_r^{-}=\{x\in K\mid v(1-x)>r\}$, and
also $U_0^{-}=1+\Mval$.

A model of ACVF is a non-trivially valued algebraically closed field.

A model of PL$_0$ is a valued field with residue 
field pseudo-finite of characteristic $0$ and
value group a $Z$-group.

We will assume that for every cardinal there is a larger inaccessible cardinal.
We will work in a monster model $K$
which for definiteness we take to be a saturated model of cardinality $\kappa$ for $\kappa$
an inaccessible cardinal. A small set is a set of cardinality smaller than $\kappa$.

This monster model can be replaced by a 
$\kappa$-saturated strongly $\kappa$-homogeneous model, for sufficiently large $\kappa$ 
depending on the arguments used,
so this assumption on inaccessible cardinals is unnecessary. In fact, for the arguments
in the paper $\omega$-saturated is enough, as long as one is careful
in the statements with the use of ``type-definable'' and ``small cardinality''.
To avoid this straightforward but cumbersome bookkeeping we prefer
to work with a monster model as described here.

Definable means definable with parameters, type-definable means a small intersection
of definable sets. If $A$ is a small set of parameters, $A$-definable means
definable with parameters in $A$ and $A$-type-definable is an (automatically small) intersection of 
$A$-definable sets.

An ind-definable set is a small directed colimit of definable sets with definable transition
maps,
and a map of ind-definable sets is ind-definable if the composition with a canonical
map in the domain factors as a composition of a definable map with a canonical map
in the target.
Ind-definable sets appear in 
Proposition \ref{unicovell}, the only ones we need are countable 
disjoint unions of definable sets. 

If $\Gamma$ is an ordered abelian group and $\gamma\in \Gamma$, then
$o(\gamma)$ is the set of elements $x\in \Gamma$ such that $|nx|<\gamma$ for all 
$n\in \mathbb{Z}_{>0}$ and $O(\gamma)$ is the set of elements such that 
$|x|<n\gamma$ for some $n\in \mathbb{Z}$. We also denote $C(\gamma)=O(\gamma)/\mathbb{Z}\gamma$, the non-standard analogue of addition modulo $\gamma$.

If $K$ is a valued field then for $a\in K^{\times}$ we use the shorthand
$O(a)=v^{-1}O(v(a))$ and $o(a)=v^{-1}o(v(a))$.
We also denote $H(a)=O(a)/o(a)$.

\subsection{Valued fields}\label{valsec}
For a valued field $K$ the valued field language consists of a two sorted language,
with a sort for $K$ and a sort for $\Gamma_{\infty}$, the ring language on $K$,
the ordered group language on $\Gamma_{\infty}$ with a constant for $\infty$ and 
the function $v:K\to \Gamma_{\infty}$. We will use the well known fact 
that ACVF eliminates quantifiers in this language, and that the completions of ACVF
are determined by fixing the characteristic and fixing the characteristic of the residue 
field.
We will also use the fact that ACVF eliminates quantifiers in the language
of valued fields with an additional sort $k$ for the residue field, 
the language of rings on $k$, and a map
$K^2\to k$ which to $(a,b)$ associates the residue of $a/b$ in case $b\neq0$ and $a/b\in\Oval$ and which is $0$ otherwise. For a proof
of this see for example \cite[Theorem 2.1.1]{HHM}.

Moreover there is a description of the definable sets $X\subset K$ due to Holly. Recall that 
a Swiss cheese in $K$ is a set of the form $B\setminus C_1\cup\dots C_n$, where $B$ is a ball, the whole of $K$ or a singleton, and where $C_i\subsetneq B$ is a ball
or a singleton for all $i$.

\begin{prop}[\cite{holly-art}]\label{holly}
Let $K\models$ACVF and $X\subset K$ be a definable subset. Then $X$ is a disjoint union of Swiss cheeses.
\end{prop}

If $K$ is a henselian valued field 
of residue characteristic $0$, then the theory
of $K$ is determined by the theory of the residue field $k$ as a pure field and
the theory of $\Gamma$ as an ordered group. This is the famous Ax-Kochen-Ershov Theorem (see, e.g., \cite{Dries-book}).

If $K$ is a henselian valued field of residue characteristic $0$ then there is 
quantifier elimination of the field sort relative to RV. This is a result of Basarab \cite{Bas-EQHens}.
We will use this in the form below.
\begin{prop}\label{qerelrv}
If $K$ is a henselian valued field of residue characteristic $0$ then every definable set
$X\subset K^n$ is a boolean combination of sets quantifier-free definable in the 
valued field language
and sets of of the form $f^{-1}rv^{-1}(Y)$ for 
$Y\subset RV^m$ definable and $f:K^n\to K^m$ given by polynomials.
Here $RV$ is given the language that interprets the short exact sequence
$1\to k^{\times}\to RV\to \Gamma\to 0$, the group language, the ring language on 
$k^{\times}$ and the ordered group language on $\Gamma$.
\end{prop}
For a proof see for example \cite[Proposition 4.3]{flenner}. 
Flenner proves Proposition \ref{rv} below, but a feature of the proof
is that it establishes Proposition \ref{qerelrv} at the same time.

We have a description of the definable sets $X\subset K$ due to Flenner.

\begin{prop}[\mbox{\cite[Proposition 5.1]{flenner}}]\label{rv}
Suppose that $K$ is a henselian valued field of residue characteristic $0$ and $X\subset K$ 
is a definable set. Then there are $a_1,\dots, a_n\in K$ and a definable set 
$Y\subset RV_{0}^n$ such that 
$X=\{x\in K\mid (rv(x-a_1),\dots, rv(x-a_n))\in Y\}$.
\end{prop}

\subsection{Definable subsets of RV}\label{rvsec}
\begin{sit}\label{rvsit}
Here we consider the following situation. 
We have $0\to A\to B\to C\to 0$ a short exact sequence of abelian groups 
such that $C$ is torsion free
and $B/nB$ is finite for every $n\geq1$.
The groups $A$ and $C$ come equipped with a language that extends that of groups.
In this case we endow $B$ with the language consisting of three sorts for $A$, $B$ and $C$,
the group language $(+,0,-)$ in $B$, the functions between sorts $A\to B$ and $B\to C$
and the extra constants, relations and functions on $A$ and $C$, and one predicate
for each $n$, interpreted as $nB$.
\end{sit}
The motivating example for the previous situation is the sequence
$1\to k^{\times}\to RV\to\Gamma\to 0$ in the case $\Gamma/n\Gamma$ and $k^{\times}/(k^{\times})^{n}$ are finite. 
In this case $k$ comes equipped with the ring language and $\Gamma$ with the ordered
group language.
\begin{prop}[\mbox{\cite[Section 3]{csqeses}}]\label{qeses}
In Situation \ref{rvsit}, we have that every definable subset of $B$ is a boolean
combination of cosets of $nB$ for some $n$, translates of definable subsets of $A$, and inverse images
of definable subsets of $C$.
\end{prop}

\subsection{Definably amenable groups in $\text{NTP}_2$ theories}\label{mopsec}
In this section we state the theorem that will allow us to relate a definable 
abelian group to a subgroup of an algebraic group.

We refer to \cite{NTP2} for the definition of NTP$_2$ theories. It is shown there for instance
that simple theories and NIP theories are NTP$_2$. An overview of NIP
theories may be found in \cite{nip}, and an overview of simple theories in \cite{simple}.
It is known that ACVF is NIP, and that the theory of pseudo-finite fields
is simple. Also a henselian valued field of residue 
characteristic $0$ is NTP$_2$ if and only if
the residue field as a pure field is (see \cite{NTP2}), so a model of PL$_0$
is NTP$_2$.

A field with extra structure is said to be algebraically bounded if in any model of its theory 
the algebraic closure in the sense of model theory coincides with the algebraic closure in the sense of fields.
A model of ACVF 
is algebraically bounded and a non-trivially valued henselian valued field 
of characteristic $0$ is algebraically bounded (see \cite{algebraically-bounded} for the henselian case). An algebraically bounded field
has a notion of dimension on definable sets, defined as the maximum of the transcendence degree of a tuple in the set over a small subfield over which the set is defined. 
We omit the proof of the following basic observation.
\begin{lema}\label{dim-alg}
Let $K$ be an $\omega$-saturated field with 
extra structure which is algebraically bounded.
Suppose that $V$ is an algebraic variety over $K$ and 
$Y\subset V(K)$ is a 
definable subset of dimension $n$. 
We can identify $V(K)$ with a subset of $V$. 
If $W\subset V$ is the subvariety of V given by the Zariski closure of $Y$ in $V$,
then the dimension of $W$ is $n$.
\end{lema}

A definable group $G$ 
is definably amenable if there is a finitely additive left invariant probability
measure on the collection of definable subsets of $G$. A group is amenable if there exists
such a measure on the collection of all subsets of $G$, so an amenable group is definably
amenable. An abelian group is amenable, see for example 
\cite[Theorem 449C] {measure}.

The next result is the main theorem of this section.
\begin{prop}
Let $K$ be a field with extra structure which is NTP$_2$ and algebraically bounded.
Let $G$ be a definably amenable group in $K$. Then there exists 
a type-definable subgroup $H\subset G$ of small index in $G$,
an algebraic group $L$ over $K$ and a type-definable group morphism
$H\to L(K)$ with finite kernel.
\end{prop}
This is \cite[Theorem 2.19]{mop}, which applies by 
\cite[Theorem 3.20]{mop}.

\subsection{Analytic functions}\label{anasec}
In this section we describe those facts about models of valued fields
with the addition of some analytic functions which we will need
in this article. These are used because various uniformization
maps, for instance the Tate uniformization map of an elliptic curve
described in Section \ref{tate-section}, are analytic functions.

For algebraically closed valued fields there is an analytic language described in
\cite{acvf-analytic}. We briefly review it here.
Let $K_0=\mathbb{F}_p^{alg}((t))$ or $\mathbb{C}((t))$ or 
the maximal unramified extension of
 $\mathbb{Q}_p$ with the $t$-adic or the $p$-adic valuation, respectively. 
Let $K_1$ be the completion of the algebraic closure of $K_0$, 
let $\Oval_1$ and $\Oval_0$ be the valuation ring of
$K_1$ and $K_0$, respectively.
A separated power series --- in two groups of variables $x=(x_1,\ldots,x_n)$ and $y=(y_1,\ldots,y_m)$ --- is a power series $f\in K_1[[x,y]]$, say
$f=\sum_{\mu,\mu}a_{\nu,\mu}x^\nu y^\mu$, which satisfies the following properties:
\begin{enumerate}
\item For any fixed multiindex $\mu$ one has $a_{\nu,\mu}\to 0$ as $|\nu|=\nu_1+\cdots+\nu_n \to \infty$.
\item There is $a_{\nu_0,\mu_0}$ with $v(a_{\nu_0,\mu_0})=\min_{\nu,\mu}v(a_{\nu,\mu})$.
\item For such an $a_{\nu_0,\mu_0}$ there are $b_k\in \Oval_1$ with $b_k\to 0$ such that
$a_{\nu,\mu}a_{\nu_0,\mu_0}^{-1}\in \Oval_0\langle b_k\rangle_k$, for all $\nu,\mu$.
\end{enumerate}
Here $\Oval_0\langle b_k\rangle _k$ denotes the closure of $\Oval_0[b_k]_k$ in $\Oval_1$.

The set of separated power series is denoted by $K_1\langle x\rangle [[y]]_s$.
Every separated power series $f(x,y)$ produces a map
$f:\Oval_1^n\times \Mval_1^m\to K_1$. If we add to the valued field language 
a function symbol for every separated power series
and interpret it as $f$ in $\Oval_1^n\times \Mval_1^m$ and $0$ outside, then we obtain a language
and a model in this language.

We call a model of the theory of $K_1$ in this language an algebraically closed valued 
field with analytic functions.

\begin{prop}\label{c-minimal}
Let $K$ be a model of ACVF with analytic functions, 
if $X\subset K$ is definable in the analytic language
then it is definable in the valued field language.
\end{prop}
This property is known as C-minimality. For a proof, see
\cite{1danacvf}. Note that the authors of that article do not
work with $\omega$-saturated models, so they need to prove
an equivalent property of C-minimality which is uniform in definable
families. 
 
Analytic languages have been defined in other contexts. We describe another such 
context we use here. Let $k_1$ be a field of characteristic $0$ and take
$K_1=k_1((t))$ with the $t$-adic valuation. We can consider power series
$f\in K_1[[x]]$ of the form $f(x)=\sum_\nu a_\nu x^\nu$ where $a_\nu\to 0$ in the $t$-adic
topology for $|\nu|=\nu1+\cdots+\nu_n\rightarrow\infty$. If we add to the valued field language a function symbol for each $f$, 
interpreted as $f(x)$ inside $\Oval_1^n$ and as $0$ outside, then we have a language 
and a model in it.

We call a model of the theory of $K_1$ in this 
language a henselian valued field of residue characteristic $0$
and value group a $Z$-group with analytic functions.
 
Just as in Proposition \ref{c-minimal} one has the following.
\begin{prop}
If $K$ is a henselian valued field of residue characteristic $0$ and value group a $Z$-group with analytic functions,
then a subset $X\subset K$ definable in the analytic
language is definable in the valued field language
\end{prop}
\begin{proof}
This situation fits into the general framework in 
\cite[Section 4.1]{analytic-cluckers},
where, as in \cite[Example 4.4 (2)]{analytic-cluckers} one takes 
$A_{m,n}=A_{m,0}=K_1\langle \eta\rangle$. 
In \cite[Theorem 6.3.7]{analytic-cluckers} a valued field quantifier elimination is proven in
a certain language we will not describe here.
From that relative quantifier elimination it follows that if $X\subset RV^n$ is definable
in the analytic language, then it is definable in the algebraic language.
Finally the framework in \cite{analytic-cluckers} satisfies the conditions of Hensel
minimality described 
in \cite{hensel-minimal}, see \cite[Theorem 6.2.1]{hensel-minimal},
and so one has that every analytically definable set $X\subset K$ is the inverse
image of $Y\subset RV_0^n$ under some map of the form 
$x\mapsto (rv(x-a_1),\dots,rv(x-a_n))$ and so it is definable in the valued field language.
\end{proof}
When $K$ is a model of ACVF 
we will only need to consider $K_1=K_0^{alg}$ for $K_0$ a complete discrete valuation field 
and the analytic
functions taken with all coefficients in a finite extension of $K_0$. These power
series seem algebraically simpler, as only complete discrete valuation rings appear.
It is not clear however that for example $K_1$ is an elementary submodel of its completion
in this analytic language, or that $K_1$ is C-minimal. In residue characteristic
$0$, $K_1$ is C-minimal by \cite{hensel-minimal}.
\subsection{Elliptic curves}\label{ellsec}
Next we review some properties of elliptic curves. 
We use valued fields which are typically not discretely valued or even of rank 1, 
so we hope
this section is useful even for readers familiar with elliptic curves.

An elliptic curve over a field $K$ is given by an equation of the form
\[y^2+a_1xy+a_3y=x^3+a_2x^2+a_4x+a_6\]
Where $a_i\in K$. This equation is assumed to have discriminant non-zero. The discriminant
of the equation is a certain polynomial with integer coefficients in the $a_i$, 
see \cite[Section III.1]{silverman-1}.

 Taking the projective closure one obtains
a projective curve $E\subset P^2_K$. In any field extension of $K$ there is only
one point at infinity $[0:1:0]$. 

This curve is a commutative algebraic group, 
the identity is the point at infinity, 
see  \cite[section III.2]{silverman-1} for 
the equations of the group law.

If $E(K)$ is the set of $K$-points, then this is a definable group in any language that extends the ring language.
If $u\in K^{\times}$ and $r,s,t\in K$ 
then we have an isomorphism of algebraic groups $E'\to E$
determined by the equations $x=u^2x'+r$ and $y=u^3y'+u^2sx'+t$, 
the corresponding equation
in $E'$ is given by coefficients $a_i'$ given in
\cite[Table 3.1]{silverman-1}.
With the relation on discriminants being of note, $u^{12}\Delta'=\Delta$.
If $K$ is a valued field then there is a change of variables 
such that $E'$ is defined by an equation with coefficients in $\Oval$. 

Now suppose $K$ is a valued field, and suppose given a Weierstrass equation with 
coefficients in $\Oval$.
The corresponding homogeneous equation
\[y^2z+a_1xyz+a_3yz^2=x^3+a_2x^2z+a_4xz^2+a_6z^3\]
produces a closed sub-scheme $\mathcal{E}\subset P^2_O$ with generic fiber $E$.
From the properness of $\mathcal{E}$, 
 we have $\mathcal{E}(\Oval)=E(K)$.
Reducing the equations mod $\Mval$ we obtain $\tilde{E}=\mathcal{E}\times_{\Oval} k\subset P^2_k$,
and a map $E(K)=\mathcal{E}(\Oval)\to \tilde{E}(k)$.
We denote by $\tilde{E}_0$ the smooth part of $\tilde{E}$.
If the discriminant of the Weierstrass equation has valuation $0$
then $\tilde{E}$ is an elliptic curve over $k$, so $\tilde{E}_0=\tilde{E}$, and in
this case we say the given Weierstrass equation has good reduction.
 Otherwise 
$\tilde{E}$ has only one singular point
and it is $k$-rational, in this case $\tilde{E}_0$ is an algebraic
group over $k$ isomorphic to the additive or a possibly twisted multiplicative group,
see \cite[Proposition III.2.5]{silverman-1} and the remark that follows, 
for a definition
of the group law and a proof. 
If $\tilde{E}_0$ is isomorphic to the additive group we say the equation has 
additive reduction. If it is isomorphic to the multiplicative group we say
the equation has split multiplicative reduction. If it is isomorphic to a twisted
multiplicative group, we say the equation has a non-split multiplicative reduction.

We denote by $E_0(K)$ the inverse image of $\tilde{E}_0(k)$ under the reduction map
$E(K)\to \tilde{E}(k)$. Then $E_0(K)$ is a subgroup of $E(K)$ and the map
$E_0(K)\to \tilde{E}_0(k)$ is a group homomorphism, see  \cite[Proposition VII.2.1]{silverman-1}. 
Note that this proposition is stated for complete discrete
valuation rings, but the proof works for any valued field.
Alternatively, one can prove this by observing that it is enough
to prove it for the algebraic closure of $K$, and then use elementary
equivalence in ACVF to see that it is enough to consider the algebraic closure
of a $\mathbb{F}_p((t)), \mathbb{Q}_p $ or $\mathbb{C}((t))$; 
and finally see that as the objects
use only finitely many parameters then it is enough to see the property
for finite extensions of the complete discrete valuation fields described, 
which are themselves complete discrete valuation fields.

It is easy to see directly that 
if the field $K$ is henselian then the map $E_0(K)\to \tilde{E}_0(k)$ 
is surjective, we omit this verification.
In any case we will call the kernel $E_0^{-}(K)$.

Inspecting the definition one gets that $E_0^{-}(K)$ consists of the point at infinity 
together with the set of pairs $(x,y)\in K$ satisfying the Weierstrass equation, such that
$v(x)<0$ or equivalently $v(y)<0$. If $(x,y)\in E_0^{-}(K)$ then from the Weierstrass equation one gets $2v(y)=3v(x)$, so that if we let $\gamma$ be such that  $-6\gamma=2v(y)=3v(x)$, we have
$v(y)=-3\gamma$, $v(x)=-2\gamma$ and $\gamma=v(\frac{x}{y})\in \Gamma$.

\begin{lema}\label{filtration-elliptic}
Assume $K$ is a valued field, and we have a Weierstrass equation over $\Oval$.
Then for the set $E_0^{-}(K)$ described above the map $E_0^{-}(K)\to M$ given by 
$(x,y)\mapsto -\frac{x}{y}$ and sending the point at infinity to $0$ is injective,
and satisfies that
\begin{enumerate}
\item $f(a+b)=f(a)+f(b)+s$ for an $s$ with $v(s)\geq v(f(a))+v(f(b))$.
\item $f(-a)=-f(a)+r$ for an $r$ with $v(r)\geq 2v(f(a))$.
\end{enumerate}
If $K$ is henselian the map is surjective.
\end{lema}
\begin{proof}
To see that the map is surjective when $K$ is henselian one can do the change of variables
$z=\frac{x}{y}$ and $w=yz^3=xz^2$, so that $x=wz^{-2}$ and $y=wz^{-3}$.
In this case the Weierstrass equation becomes
\[w^2+a_1zw^2+a_3z^3w=w^3+a_2z^2w^2+a_4z^4w+a_6z^6\]
Reducing mod $\Mval$ we get the equation
$w^2=w^3$ which has a unique non-zero solution at $w=1$, so we may apply Hensel's lemma.
This also shows the map is injective.

To see the two displayed equations we may replace $K$ by its algebraic closure,
then because this is first order expressible we may reduce to the case in which
$K$ is the algebraic closure of $\mathbb{Q}_p$ or $\mathbb{C}((t))$ or $\mathbb{F}_p((t))$.
In this case the Weierstrass equation has coefficients in a finite extension of the fields
mentioned and $a$ and $b$ also, so we may assume $K$ is such a finite extension.
In particular it is a complete discrete valuation field.

In this case the discussion in \cite[Section IV.1]{silverman-1} shows that
for $z_1,z_2\in \Mval$ we have $f(f^{-1}(z_1)+f^{-1}(z_2))=z_1+z_2+z_1z_2F(z_1,z_2)$
for $F(T,S)$ a power series with coefficients in $\Oval$.
Similarly one has $f(-f^{-1}(z))=-z+z^2g(z)$ where $g(S)$ is a power series with 
coefficients in $\Oval$. 
See also \cite[Remark IV.2.1]{silverman-1} and 
\cite[Example IV.3.1.3]{silverman-1}.
\end{proof}
If $K$ is any valued field, and we have a given Weierstrass equation with coefficients
in $\Oval$, then inside $E_0^{-}$ and for $a\in\Gamma$ with $a>0$ we have the sets $E_a$
given by the pairs $(x,y)$ such that $v(\frac{x}{y})\geq a$ 
together with the point at infinity, 
and $E_a^{-}$ the set of pairs $(x,y)$ such that $v(\frac{x}{y})>a$ together with the point
at infinity.
\begin{lema}\label{filtration-elliptic-subquotient}
With the above notation
$E_a$ and $E_a^-$ are subgroups of $E(K)$ and the map
$E_a\to B_a/B_a^-$ given by $(x,y)\mapsto -\frac{x}{y}$ and which sends the 
point at infinity to $0$, is a group map. 
When $K$ is henselian it is surjective.
\end{lema}
\begin{proof}
This follows easily from Lemma \ref{filtration-elliptic}.
\end{proof}
\begin{defi}\label{minimal-divisible}
If $K$ is a valued field and $E$ is an elliptic curve
then a Weierstrass equation of $E$ will be called minimal
if the valuation of the discriminant is minimal among all possible 
Weierstrass equations for $E$.
\end{defi}
Such an equation need not exist.

Assume that $K$ is a valued field elementary equivalent as a valued field to a 
discrete valuation field. In this case given an elliptic curve $E$ over $K$, a minimal
Weierstrass equation  exists because this is first order expressible and true
in a discretely valued field.

Given a minimal Weierstrass equation, another Weierstrass equation
related to it by change of variables is minimal if and only if we can obtain one
from the other through a change of variables 
$x=u^2x'+r$ and $y=u^3y'+u^2sx'+t$, with
$u\in \Oval^{\times}$ and $r,s,t\in \Oval$, see
 \cite[Proposition VII.1.3]{silverman-1}. This proposition is stated
for complete discrete valuation rings, but the proof is valid for any valued
field.

In case the residue characteristic is not $2$ or $3$ then there is a change of variables
with $u\in \Oval^{\times}$ and $r,s,t\in \Oval$ (coming from completing the square and cube)
that produces a minimal Weierstrass equation of the form 
\begin{equation}\label{wei-sim} y^2=x^3+Ax+B
\end{equation}
See \cite[Section III.1]{silverman-1} for details.
Two such Weierstrass equations are related through the change of variables
$x=u^2x'$ and $y=u^3y'$, 
and two minimal Weierstrass equations through a $u\in \Oval^{\times}$.
In this case the coefficients of $E'$ are given by 
$u^4A'=A$ and $u^6B'=B$.

It follows from this discussion that if $K$ has residue field
of characteristic not $2$ or $3$ and the value group is a $Z$-group then 
a minimal Weierstrass equation exists and can be found in the form
\eqref{wei-sim}, and an equation of that form with $A,B\in \Oval$ is minimal
 if and only if $v(A)<4$ or $v(B)<6$.

If $K$ has residue field of characteristic not $2$ or $3$ and the value group
is $2$- and $3$-divisible, then the discussion above also shows that a minimal
Weierstrass equation exists and can be found in the form \eqref{wei-sim}
and a Weierstrass equation of that form with $A,B\in \Oval$
is minimal if and only
if $v(A)=0$ or $v(B)=0$.

If we have a valued field $K$ and an elliptic curve $E$ over $K$,
we say $E$ has good, additive, split
multiplicative or non-split multiplicative reduction if a minimal Weierstrass equation
of $E$ does, so when we speak of the reduction type of an elliptic curve
by convention we assume a minimal Weierstrass equation for it exists.

 Similarly the notations $E_0$, $E_0^-$, $E_r$ and $E_r^-$ are understood
to be relative to a minimal Weierstrass equation for $E$. 

We mention that in case $K$ is a valued field with residue characteristic
not $2$ or $3$ and $2$- and $3$-divisible value group
 then an elliptic curve over $K$ has good or
 multiplicative reduction, by the criterion in 
\cite[Proposition III.2.5]{silverman-1}.

The following two lemmas are a relatively minor point.
\begin{lema}\label{smooth-henselian}
Suppose $(R,\Mval,k)$ is a henselian local ring and suppose
$f:X\to Y$ is a smooth map of schemes over $R$. Then for every $x\in X(k)$ and 
$y\in Y(R)$ such that $\bar{y}=f(x)$, there exists $z\in X(R)$ such that
$\bar{z}=x$ and $f(z)=y$.
\end{lema}
This lemma is well known and we omit the proof.
\begin{lema}\label{multiplication-by-n-henselian}
Let $K$ be a henselian valued field of residue of characteristic $0$ with group
a $Z$-group.
Assume given an elliptic curve $E$ over $K$. Then $nE_0(K)=\pi^{-1}n\tilde{E}_0(k)$,
where $\pi:E_0(K)\to \tilde{E}_0(k)$ is the reduction map.
\end{lema}
The previous lemma should be true in more general henselian valued fields,
starting with a Weierstrass equation with coefficients in $\Oval$, but we do not know a proof.
\begin{proof}
This property is first order expressible so without loss of generality $K=k((t))$ is a 
complete discretely valued field. In this case we get that $\mathcal{E}_0\subset \mathcal{E}$,
the set of smooth points is a group scheme over $\Oval$ with generic fiber $E$ and with
$E_0(K)=\mathcal{E}_0(\Oval)$, see 
\cite[Corollary IV.9.1]{silverman-2} and \cite[Corollary IV.9.2]{silverman-2}.
Recall that for any group scheme $X$ over a scheme $S$ which is smooth the multiplication
map is smooth, and so the multiplication by $n$ map is smooth $\mathcal{E}_0\to \mathcal{E}_0$. Then this becomes a consequence of Lemma \ref{smooth-henselian}.
\end{proof}

\subsection{Tate uniformization}\label{tate-section}
In this section we review the Tate uniformization of an elliptic curve.
We will be interested in the case where $K$ is an algebraically closed field
of residue characteristic not $2$ or $3$ and when $K$ is a pseudo-local field of 
residue characteristic $0$ so we treat these two cases separately.

Assume first that $K$ is an $\omega$-saturated algebraically closed valued 
field of residue characteristic not 
$2$ or $3$ with analytic functions as in Section \ref{anasec}. 
Define $K_0$ to be $\mathbb{F}_p^{alg}((t))$ or $\mathbb{C}((t))$ or 
the maximal unramified extension of 
$\mathbb{Q}_p$
according to the characteristic of $K$ and the residue field of $K$, 
and let $K_1$ be the completion of the algebraic closure of $K_0$.
Now there are functions
$a_4:\Mval_1\to K_1$ and $a_6:\Mval_1\to K_1$ given by a separated power series in one
variable $K_1[[q]]_s$, defined in 
\cite[Theorem V.3.1]{silverman-2}. 
In fact they are given by elements of $\mathbb{Z}[[q]]$.
Then one defines the elliptic curve $E_q$ with equation 
$y^2+xy=x^3+a_4(q)x+a_6(q)$. This equation has discriminant different from $0$ for all
$q\in \Mval_1$.

One has functions $X(u,q)$ and $Y(u,q)$ defined as 
\[X(u,q)=\frac{u}{(1-u)^2}+\sum_{d\geq 1}(\sum_{m|d}m(u^m+u^{-m}-2))q^d\]
\[Y(u,q)=\frac{u^2}{(1-u)^3}+\sum_{d\geq 1}(\sum_{m|d}\frac{(m-1)m}{2}u^m-\frac{m(m+1)}{2}u^{-m}+m)q^d\]
For $-v(q)<v(u)<v(q)$ and $u\neq 1$. This produces a map $v^{-1}(-v(q),v(q))\to E_q(K_1)$
that takes $u=1$ to the point at infinity, and $u\neq 1$ to 
$((X(u,q),Y(u,q))$. This maps extends in a unique way to a surjective group map 
$K_1^{\times}\to E_q(K_1)$ with kernel $q^{\mathbb{Z}}$. 

This is proven in 
\cite[Theorem V.3.1]{silverman-2}. This theorem 
is stated for $K$ a $p$-adic field, but the proof works equally well
for a complete rank one valued field. 
Note that \cite[Remark V.3.1.2]{silverman-2}
states that the elementary proof given is valid in this generality except
for the proof of surjectivity, but in fact the proof of 
surjectivity also works
in the more general context with some small modifications detailed below.
We use this elementary proof instead of the proof using $p$-adic analytic
methods in \cite{roquette} mentioned in \cite[Remark V.3.1.2]{silverman-2}
as we will make use of some extra information obtained in the course of the
proof, for the proof of Proposition \ref{tate-uniformization}(6).

Now we review the proof of \cite[Theorem V.3.1]{silverman-2}.
The maps $X$ and $Y$ are definable in the analytic language.
Indeed $X(u,q)=L_1(u)+L_2(uq,q)+L_3(u^{-1}q,q)$
where $L_1(u)=\frac{u}{(1-u)^2}$, and $L_2$, $L_3$ are given by separated power series
in two variables $K_1[[v,q]]_s$. 
Indeed, $L_2(v,q)=\sum_{k\geq 1, m\geq 1} mv^mq^{m(k-1)}$, and 
$L_3(v,q)=\sum_{k\geq 1, m\geq 1}(mv^mq^{m(k-1)}-2q^{mk})$ are even given by power series
with coefficients in $\mathbb{Z}$. A similar consideration applies to $Y$.
The map $(X,Y)$ produces an injective group map $K_1^{\times}/q^{\mathbb{Z}}\to E_q(K_1)$,
this follows from formal identities on $\mathbb{Q}(u)[[q]]$.

For the Tate map $t:K_1^{\times}\to E_q(K_1)$ one has
$t(\Oval_1^{\times})\subset E_{q,0}(K_1)$ and $t(1+\Mval_1)\subset E_q^-(K_1)$.

Under the bijections $\Mval_1\to 1+\Mval_1$ and $E_q^-(\Mval_1)\to \Mval_1$ described in 
Lemma \ref{filtration-elliptic} we get a map
$\phi:\Mval_1\to \Mval_1$ given by $x\mapsto x(1+\sum_{m\geq 1}\gamma_m x^m)$, and 
the coefficients belong to the ring $\Oval_0\langle q\rangle$ described in 
Section \ref{anasec}. So the argument in 
\cite[Theorem V.3.1]{silverman-2} shows that
$\phi$ is surjective, and the inverse is given as another such power series. 

Also one gets that $t(\Oval_1^{\times})=E_{q,0}(K_1)$
by reducing mod $1+\Mval_1$ and $E_{q,0}^-(K_1)$ respectively.

After this one decomposes $E_q(K_1)\setminus E_{q,0}(K_1)$ 
into some disjoint sets $U_r, V_r, W$ for 
$0<r<\frac{1}{2}v(q)$, defined as
\[U_r=\{(x,y)\in E_q(K_1)\mid v(y)>v(x+y)=r\}\]
\[V_r=\{(x,y)\in E_q(K_1)\mid v(x+y)>v(y)=r\}\]
\[W=\{(x,y)\in E_q(K_1)\mid v(y)=v(x+y)=\frac{1}{2}v(q)\}\]
See \cite[Lemma V.4.1.2]{silverman-2}, its proof generalizes to our context
with minimal notational modifications.
Further the proof of \cite[Lemma V.4.1.4]{silverman-2} generalizes
to our context with minimal notational modifications to show 
that each of these sets belong to the same
$E_{q,0}(K_1)$ coset, and one sees in the proof that $V_r=-U_r$.

Now we claim that if $0<v(u)=r<\frac{1}{2}v(q)$ then
$t(u)\in U_r$.
Indeed in the formula for $X(u,q)$ one has that if 
$a_d=\sum_{m|d}m(u^m+u^{-m}-2)$ then $v(a_d)\geq -dr$ and so
$v(a_dq^d)\geq -dr+dv(q)>dr\geq r=v(\frac{u}{(1-u)^2})$, so we conclude
$v(X(u,q))=r$.
A similar computation shows all the summands in $Y$ have valuation $>r$ so one gets
$v(Y(u,q))>r$ and $t(u)\in U_r$ as required.

Next we claim that if $r=v(u)=\frac{1}{2}v(q)$ then $t(u)\in W$.
Indeed, in this case the computation above shows that for the $d$ term in $X$, 
$v(a_dq^d)\geq dr>r$ if $d>1$, so modulo $B_r^-$ we have that
$X$ is $\frac{u}{(1-u)^2}+qu^{-1}$.
Similarly, modulo $B_r^{-}$ we have that $Y$ becomes
$-qu^{-1}$. So we obtain $v(X(u,q)+Y(u,q))=r=v(Y(u,q))$ as required.

Now since the sets 
\[U_r'=\{u\in K_1\mid v(u)=r\}\]
\[V_r'=\{u\in K_1\mid v(u)=-r\}\]
\[W'=\{u\in K_1\mid v(u)=\frac{1}{2}v(q)\}\]
for $0<r<\frac{1}{2}v(q)$ 
are $\Oval_1^{\times}$-cosets satisfying $V_r'^{-1}=U_r'$ we conclude that 
$U_r,V_r$ and $W$ are $E_{q,0}^{-}(K_1)$ cosets satisfying 
$t(U_r')=U_r$, $t(V_r')=V_r$ and $t(W')=W$, and $t$ is surjective, as required.

Now by elementary equivalence we get that 
the functions $a_4$ and $a_6$ produce functions denoted by the same symbol 
$a_4,a_6:\Mval\to K$, and we obtain for every $q\in \Mval$ an elliptic curve $E_q$ and a function
$(X,Y):v^{-1}(-v(q),v(q))\setminus\{1\}\to E_q$ that extends to a surjective group homomorphism
$O(q)\to E_q$ with kernel $q^{\mathbb{Z}}$. If further $E$ has multiplicative reduction
then there is a unique $q$ 
such that $E$ is isomorphic to $E_q$ over $K$. See 
\cite[Theorem V.5.3]{silverman-2}, where it is stated in less generality,
 but the proof works equally well for 
a complete valued field of rank 1 and residue characteristic not 2 or 3.

We collect the previous discussion into a theorem.
\begin{prop}\label{tate-uniformization}
Suppose $K$ is a model of ACVF
 with analytic functions and residue field
of characteristic not $2$ or $3$ and let $E$ be an elliptic curve over $K$ with 
multiplicative reduction.
Then there is a surjective group homomorphism $t:O(q)\to E(K)$ with kernel $q^\mathbb{Z}$
ind-definable in the analytic language. 
We have further that:
\begin{enumerate}
\item $v(q)$ equals the valuation of the discriminant of a minimal 
Weierstrass equation of $E$.
\item The image of $\Oval^{\times}$ is $E_0$.
\item The image of $U_0^-$ is $E_0^-$.
\item For the bijections $f:U_0^-\to \Mval$ and $g:E_0^-\to \Mval$ given by 
$f(x)=x-1$ and $g(x,y)=-\frac{x}{y}$ we have 
$v\circ g\circ t=v\circ f$. 
\item The image of $U_r$ is $E_r$ and of $U_r^-$ is $E_r^-$ for $r\in \Gamma_{>0}$.
\item The composition $E(K)\to O(q)/q^{\mathbb{Z}}\to O(v(q))/\mathbb{Z}v(q)$ is definable
in the valued field language.
\item If $K$ is $\omega$-saturated 
the image of $o(q)$ in $E(K)$ is type-definable in the valued field language.
\end{enumerate}
\end{prop}
\begin{proof}
The properties 1-5 
of the Tate uniformization map described here are all first order expressible
so they follow from the corresponding result in $K_1$.

For property 6, let
 $E$ be an elliptic curve over $K$ with multiplicative reduction, understood
as in the remark following Definition \ref{minimal-divisible}, then $E$ is isomorphic
to $E_q$ for a $q\in M$.
We conclude that if $E$ has multiplicative reduction then it has a Weierstrass equation
of the form $y^2+xy=x^3+a_4'x+a_6'$ with $v(a_4')>0$ and $a_6'=a_4'+c$ 
with $v(c)\geq 2v(a_4)$.
Such an equation is minimal in the sense of Definition \ref{minimal-divisible}.
Now assume that $E_q$ is isomorphic to such an elliptic curve. Then $E$ is obtained from
$E_q$
via a change of variables $x=u^2x'+r$ and $y=u^3y'+u^2sx'+t$ with 
$u\in \Oval^{\times}$ and $r,s,t\in \Oval$.
From the discriminant formula $\Delta=-a_6+a_4^2+72a_4a_6-64a_4^3-432a_6^2$, and
$u^{12}\Delta'=\Delta$ we obtain that $v(\Delta)=v(a_4')=v(q)=l$. 
Now reducing the equations in \cite[Table 3.1]{silverman-1} mod $B_l$ 
(for $c_4$ and $c_6$)
we obtain $u^4=1$ and $u^6=1$ mod $B_l$. From here we get $u^2=1$ mod $B_l$.
Replacing $u$ by $-u$ if necessary (this has the effect of composing $E_q\to E$ with
$-:E\to E$) we conclude $u=1$ mod $B_l$.  Reducing the change of variable equations
in \cite[Table 3.1]{silverman-1} mod $B_l$ (for $a_1,a_2$ and $a_3$; note that
this step uses that the residue field
has characteristic not $2$ or $3$) we obtain
$s=r=t=0$ mod $B_l$. So if $(x',y')$ is the image of 
$(x,y)\in E_q(K)\setminus E_{q,0}(K)$ under the map $E_q\to E$
then $x'=x$ and $y'=y$ mod $B_l$.

So if we define $U_r', V_r'$ and $W'$ by the same formulas as 
$U_r, V_r, W$ but using the Weierstrass equation in $E$
then the isomorphism $E_q\to E$ takes $U_r$ to $U_r'$, 
$V_r$ to $V_r'$ and $W$ to $W'$ (so in particular $U_r',V_r'$ and $W'$ are the $E_0(K)$
cosets). Further if the image of $(x,y)$ is $(x',y')$ then $v(x+y)=v(x'+y')$ in $U_r$
and $v(y)=v(y')$ in $V_r$.
So the map $w:E(K)\to (-\frac{1}{2}v(a_4),\frac{1}{2}v(a_4)]$ is given explicitly
as $w(P)=0$ if $P\in E_{0}(K)$, $w(x,y)=\frac{1}{2}v(a_4)$ if $v(y)=v(x+y)$,
$w(x,y)=v(x)$ if $0<v(y)<v(x+y)$ and $w(x,y)=-v(x)$ if $0<v(x+y)<v(y)$.

Property 7 follows from 6.
\end{proof}

In our situation we are able to prove using the explicit formulas that 
$E(K)\to O(q)/q^{\mathbb{Z}}\to O(v(q))/\mathbb{Z}v(q)$ is definable in the valued
field language.

Related to this 
we can ask whether for $K$ a model of ACVF with analytic structure
and $x\in K^n$, if $y\in \Gamma$ is definable from $x$ in
the analytic language then it is definable from $x$ in the valued field language.
Also we can ask whether if $a\in K$ is algebraic over $x$ in the valued field language
and definable from $x$ in the analytic language then it is definable from $x$ 
in the valued field language. 

The second question together with C-minimality is equivalent to saying 
that for any one-dimensional
set $X$ definable in the valued language and any set $Y\subset X$ definable in the analytic
language, $Y$ is definable in the valued field language.

The first question together with the second one implies that for any map 
$f:X\to \Gamma$ definable in the analytic language, where $X$ is one dimensional definable in the valued field language, $f$ is definable in the valued field language.

Using \cite[Proposition 2.4]{one-dimensional-p-adic} we obtain the following.

\begin{prop}\label{unicovell}
Suppose $K$ is an $\omega$-saturated 
algebraically closed valued field of residue characteristic not $2$ or 
$3$ with analytic functions 
and let $E$ be an elliptic curve over $K$ with multiplicative reduction.
Let $t:O(q)\to E$ be the Tate uniformization map of Proposition \ref{tate-uniformization}.
Then there is a group $O_E$ which is ind-definable in the valued field language and 
a group isomorphism $O_E\to O(q)$ ind-definable in the analytic language, such that the 
composition $O_E\to O(q)\to E$ is ind-definable in the valued field language.
\end{prop}

For reference we mention the following related result.
\begin{lema}[mbox{ \cite[Lemma 29]{groups-in-z}}]\label{r-extension}
Suppose $M$ is a model of any theory.
Let $G$ be a strict ind-definable abelian group
 and $T\subset G$ a type-definable subgroup.
 Let $H$ be a strict ind-definable group
and $\phi:T\to H$ a type-definable group morphism.
Suppose given a surjective group morphism
$\pi:G\to \mathbb{R}^s$ with kernel $T$. Assume that:
\begin{enumerate}
\item If $T\subset U\subset G$ is definable then there is 
$0\in V\subset \mathbb{R}^s$ open such that
$\pi^{-1}(V)\subset U$.
\item If $T\subset U\subset G$ is definable then $\pi(U)$ is bounded.
\end{enumerate}
Then $\phi$ extends to a unique ind-definable group morphism $\phi:G\to H$.
\end{lema}
We only use this when $T\subset G$ are 
$o(a)\subset O(a)$ in $K^{\times}$ or $o_E(a)\subset O_E(a)$ for an elliptic curve of
multiplicative reduction, or $o(\gamma)\subset O(\gamma)$ in $\Gamma$.

Now suppose $K$ is a henselian valued field of residue characteristic $0$ and value group 
a $Z$-group with analytic functions, as defined in Section \ref{anasec}.
Then the discussion before goes through and produces the following two statements.

\begin{prop}\label{tate-uniformization-2}
Suppose $K$ is an $\omega$-saturated 
henselian valued field of residue characteristic $0$ and 
value group a $Z$-group with analytic functions.
Let $E$ be an elliptic curve over $K$ with 
split multiplicative reduction.
Then there is a surjective group morphism $t:O(q)\to E(K)$ with kernel $q^\mathbb{Z}$
ind-definable in the analytic language. 
We have further that:
\begin{enumerate}
\item $v(q)$ equals the valuation of the discriminant of a minimal 
Weierstrass equation of $E$.
\item The image of $\Oval^{\times}$ is $E_0$.
\item The image of $U_0^-$ is $E_0^-$.
\item For the bijections $f:U_0^-\to \Mval$ and $g:E_0^-\to \Mval$ given by 
$f(x)=x-1$ and $g(x,y)=-\frac{x}{y}$ we have 
$v\circ g\circ t=v\circ f$.
\item The image of $U_r$ is $E_r$ and that of $U_r^-$ is $E_r^-$,  where $r\in \Gamma_{>0}$.
\item The composition $E(K)\to O(q)/q^{\mathbb{Z}}\to O(v(q))/\mathbb{Z}v(q)$ is definable
in the valued field language.
\item 
The image of $o(q)$ in $E(K)$ is type-definable in the valued field language.
\end{enumerate}
\end{prop}

\begin{prop}
Suppose $K$ is an $\omega$-saturated 
henselian valued field of residue characteristic $0$ and 
value group a $Z$-group with analytic functions.
Let $E$ be an elliptic curve over $K$ with split multiplicative reduction.
Let $t:O(q)\to E$ be the Tate uniformization map of Proposition \ref{tate-uniformization-2}.
Then there is a group $O_E$ which is ind-definable in the valued field language and 
a group isomorphism $O_E\to O(q)$ ind-definable in the analytic language, such that the 
composition $O_E\to O(q)\to E$ is ind-definable in the valued field language.
\end{prop}
Finally we have in the other reduction cases the following.
\begin{prop}
Let $K$ be a henselian field of residue characteristic $0$ and value group a $Z$-group.
Let $E$ be an elliptic curve over $K$ of good, additive or non-split multiplicative 
reduction. Then 
$E(K)/E_0(K)$ is a finite group of order at most $4$.
\end{prop}
\begin{proof}
This is first order expressible so without loss of generality $K=k((t))$ is a complete
discrete valuation field.
In this case the statement corresponds to \cite[Corollary IV.9.2 (d)]{silverman-2}.
\end{proof}

\subsection{Algebraic groups}\label{alggrosec}
Let $K$ be a perfect field. To set notation 
we will take algebraic groups over $K$ to be reduced finite
type group schemes $G$ over $K$. In this case the connected component of the identity
is an open sub-scheme which is a subgroup of $G$ of finite index denoted $G^0$. Often
we will identify the group $G$ with the set of points of the algebraic closure of $K$.

Examples of algebraic groups are the additive group 
$(K^{alg},+)$
, the multiplicative group 
$((K^{alg})^{\times},\cdot)$
, and the elliptic curves discussed in 
Section \ref{ellsec}.

If $K$ is a field and $d\in K^{\times}\setminus (K^{\times})^{2}$ we call
$G(d)(K)$, the subgroup of the multiplicative group $K(\sqrt{d})^{\times}$ with norm one,
a twisted multiplicative group of $K$.
$G(d)(K)$ acts on the two-dimensional $K$-vector space $K(\sqrt{d})$ so choosing the basis
$\{1,\sqrt{d}\}$ we get an embedding $G(d)\subset \mathrm{GL}_2(K)$,
$G(d)(K)=\{\begin{bmatrix}a & db\\ b & a\end{bmatrix}| a^2-db^2=1\}$.
It is now clear how to define $G(d)$ so that it is a one dimensional algebraic group over
$K$ with $K$-points as described.

We use the well known classification of one dimensional connected algebraic groups

\begin{prop}\label{one-dimensional-algebraic-group}
Suppose $K$ is a perfect field of characteristic not $2$, 
and $G$ is a connected one dimensional algebraic group
over $K$. Then $G$ is isomorphic as an algebraic group over $K$ to the additive group,
the multiplicative group, a twisted multiplicative group, or an elliptic curve.

If $K$ is an algebraically closed field of any characteristic, and $G$ is a connected
one dimensional algebraic group over $K$, then $G$ is isomorphic as an algebraic group
to the additive group, the multiplicative group, or an elliptic curve.
\end{prop}
See for example \cite[Corollary 16.16]{algebraic-groups} for the affine case and
\cite[Theorem 8.28]{algebraic-groups} 
for the general case.

In the following proposition we collect some basic statements about morphisms between
one dimensional algebraic groups.

\begin{prop}\label{alg-group-map}
Let $K$ be a field of characteristic $0$.
\begin{enumerate}
\item The only algebraic group morphism from the multiplicative group or an elliptic curve
to the additive group is the trivial one. The algebraic group morphisms 
$(K^{alg},+)\to (K^{alg},+)$ defined over $K$ are multiplication by a scalar from $K$.
\item The only algebraic group morphism from the additive group or an elliptic curve to 
the multiplicative group are trivial. The algebraic group morphisms
$f:((K^{alg})^{\times},\cdot)\to ((K^{alg})^{\times},\cdot)$ are of the form
$x\mapsto x^n$ for an $n\in \mathbb{Z}$.
\item The only algebraic group morphisms from the additive group or the multiplicative 
group to an elliptic curve are trivial. If $E$ and $E'$ are elliptic curves and 
$f:E\to E'$ is a nontrivial algebraic group morphism, then it is surjective on 
$K^{alg}$-points, and there is a non-constant group morphism $g:E'\to E$ and an integer $n$
such that $gf$ and $fg$ are multiplication by $n$ maps. The multiplication by $n$ maps (for $n>0$)
are non-trivial group morphisms.
\end{enumerate}
\end{prop}
This is well known. The third statement is \cite[Theorem III.6.1]{silverman-1}.

The following two basic observations are well known and we omit the proofs.

\begin{lema}\label{closure-group}
Let $K$ be a perfect field and $G$ an algebraic group over $K$.
Let $A\subset G(K)$ be a subgroup of $G(K)$.
Let $H\subset G$ be the reduced closed subscheme with points equal to the 
Zariski closure of $A$. Then $H$ is an algebraic subgroup of $G$.
\end{lema}
\begin{lema}\label{def-to-alg-acf}
Let $K$ be a model of ACF$_0$, and $G$ and $H$ two algebraic groups over $K$.
If $f:G\to H$ is a definable group morphism, then it is algebraic.
\end{lema}

\subsection{Almost saturated sets and maps}\label{opasec}
Given a definable function $f:X\to Y$ and a definable subset $Z\subset X$
then $Z$ is said to be almost saturated with respect to $f$ if
there is a finite set $S\subset Y$ such that if $y\in Y\setminus S$ then 
$f^{-1}(y)\subset Z$ or $f^{-1}(y)\cap Z=\emptyset$. In this case the set
$S$ is called an exceptional set of $Z$ with respect to $f$. This is the same 
as saying that $f(X)\cap f(X\setminus Z)\subset S$, so $Z$ is almost saturated
with respect to $f$ if and only if $f(X)\cap f(X\setminus Z)$ is finite.

The function $f:X\to Y$ is said to be opaque if
every definable set $Z\subset X$ is almost saturated with respect to $f$.

\begin{lema}\label{type-opa}
Let $f:X\to Y$ be a definable function. Then $f$ is opaque if and only if the map
$f_{*}:S_X(\mathcal{U})\to S_Y(\mathcal{U})$ satisfies that if $f_*(p)=f_*(q)$ is a 
non-realized type, then $p=q$.
\end{lema}
\begin{proof}
Suppose $f$ is opaque and suppose $p$ and $q$ are types such that $f_*(p)=f_*(q)$ is non-realized.
Suppose $Z\subset X$ is definable and $Z\in p$.
Then there is $S\subset Y$ finite as in the definition of opaque.
Then because $f_*(p)$ is not realized we get $f(Z)\setminus S\in f_*(p)$. So we conclude
$f^{-1}((f(Z)\setminus S))\in q$. We have by assumption that 
$f^{-1}((f(Z)\setminus S))\subset Z$, so $Z\in q$ as required.

In the other direction, 
suppose $Z\subset X$ is definable and is not almost saturated relative to $f$.
Then $E=f(Z)\cap f(X\setminus Z)$ is a definable set and is infinite.
Let $r$ be a complete non-realized global type in $E$.
Then there are complete types $p$ and $q$ in $X$ such that $p$ is in $Z$ and 
$q$ is in $X\setminus Z$, and $f_*(p)=f_*(q)=r$ as required.
\end{proof}
Now we give some simple properties of this notion.
\begin{lema}\label{opa1}
Suppose $f:X\to Y$ and $g:Y\to Z$ are definable maps.
\begin{enumerate}
\item If $f$ and $g$ are opaque then $gf$ is opaque.
\item If $gf$ is opaque and $g$ has finite fibers then $f$ is opaque.
\end{enumerate}
\end{lema}
\begin{proof}
Suppose $f$ and $g$ are opaque and take $W\subset X$ a definable set. 
Take $S\subset Y$ a finite exceptional set of $W$ with respect to $f$ and 
$T\subset Z$ a finite exceptional set of $f(W)$ with respect to $g$.
Then $f(S)\cup T$ is a finite exceptional set of $W$ with respect to $gf$.

Now suppose $gf$ is opaque and $g$ has finite fibers, and suppose $W\subset X$ is a definable
set. Take $S\subset Z$ a finite exceptional set of $W$ with respect to $gf$.
Then $g^{-1}(S)$ is a finite exceptional set of $W$ with respect to $f$.
\end{proof}
Next we give some examples of opaque maps in henselian valued fields of residue
characteristic $0$ and 
in algebraically closed valued fields of any characteristic.
\begin{lema}\label{rv-is-opa}
If $K$ is a henselian field of residue characteristic $0$ then $K^{\times}\to RV$ is opaque.
\end{lema}
\begin{proof}
Let $X\subset K^{\times}$ be definable and let $a_1,\dots, a_n\in K$ and $Y\subset RV_0^n$
as in Proposition \ref{rv}. If $x\in K^{\times}$ and $rv(x)\neq rv(a)$ then $rv(xy-a)=rv(x-a)$ for every $y\in 1+\Mval$.

So we get that $S=\{rv(a_i)\mid a_i\neq 0\}$ makes $X$ almost saturated relative to $rv$.
\end{proof}

\begin{lema}\label{opa+}
If $K$ is a henselian field of residue characteristic $0$ or 
a model of ACVF then $\Oval\to k$ is opaque.
\end{lema}

\begin{proof}
Assume $K$ is henselian of residue characteristic $0$.
In this case the lemma follows by Lemma \ref{rv-is-opa}, as
$\Oval\to k$ coincides with $rv$ except in the fiber of $0$.

Now assume that $K$ is algebraically closed.
Being almost saturated is closed under boolean combinations, so by Swiss-cheese decomposition (Proposition \ref{holly}),
it is enough to show that balls are almost saturated under the map $\Oval\to k$.
But a ball has image a point or it contains all of $\Oval$.
\end{proof}

\begin{lema}\label{fmult-opa}
Suppose $K$ is a henselian field of residue characteristic $0$ or 
a model of ACVF, and let $r\in\Gamma_{>0}$.
 Then the map
$U_r\to B_r/B_r^-$ given by $1+x\mapsto \bar{x}$ is opaque.
\end{lema}

\begin{proof}
Under the bijection $U_0^-\to \Mval$ given by $x\mapsto x-1$ this becomes the projection
$B_r\to B_r/B_r^-$, after rescaling by an element $a\in B_r\setminus B_r^-$ this becomes 
Lemma \ref{opa+}.
\end{proof}

\begin{lema}\label{mult-opa}
If $K$ is a henselian field of residue characteristic $0$ 
or an algebraically closed field then
$\Oval^{\times}\to k^{\times}$ is opaque.
\end{lema}

\begin{proof}
This is equivalent to Lemma \ref{opa+}
\end{proof}

\begin{lema}\label{opa-ell2}
If $K$ is a model of ACVF 
or a henselian valued field of characteristic $0$
and we are given a Weierstrass equation
with coefficients in $\Oval$ then the
maps $E_a\to B_a/B_a^-$ given by $(x,y)\mapsto 
-\overline{\frac{x}{y}}$ are opaque.
\end{lema}

\begin{proof}
Under the bijection of Lemma \ref{filtration-elliptic}, and after rescaling
 the result corresponds to 
Lemma \ref{opa+}
\end{proof}

\begin{lema}\label{v-opa}
Suppose $K$ is a model of ACVF. Then the map 
$v:K\to \Gamma_{\infty}$ is opaque.
\end{lema}

\begin{proof}
The condition of being almost saturated with respect to $v$ is closed under boolean combinations, so by Proposition \ref{holly} we see that we only have to prove it for $X=K$, or $X$
equal to a point or a ball.
We may assume then that $X$ is a ball. If $0\in X$ then $X=v^{-1}v(X)$ is saturated with 
respect to $v$.
Otherwise $v(X)=\{r\}$ so the condition holds too (with $S=\{r\}$).
\end{proof}

We mention that the previous result does not hold for most other valued fields, for example
$K=k((t))$
 where $k$ is a field of characteristic $0$ and $k^{\times}\neq (k^{\times})^2$, as then the set of squares $(K^{\times})^2$ is not almost saturated with respect to $v$.
 
\begin{lema}\label{opa-ell-acvf}
Let $K$ be a model of ACVF, 
and suppose given a Weierstrass equation
with coefficients in $\Oval$. Then the map
$E(K)\to \tilde{E}(k)$ is opaque.
\end{lema}

\begin{proof}
By Lemma \ref{type-opa} and by quantifier elimination 
we have to prove that if $(x,y)\in E(K')$ and $(x,y)\in E(K')$ are such that
$r(x,y),r(x',y')\notin \tilde{E}(k)$ then $K(x,y)\cong K(x',y')$ as valued fields.
This is to say, if $f(T,S), g(T,S)\in K[T,S]$ are polynomials then 
$v(f(x,y))\leq v(g(x,y))$ is equivalent to $v(f(x',y'))\leq v(g(x',y'))$.
Notice that $x,y,x',y'\in \Oval'$ because otherwise $r(x,y)$ is the point at infinity.
In this case $r(x,y)=(\pi(x),\pi(y))$, for $\pi:\Oval\to k$.

Dividing the polynomials $f,g$ by $h=S^2+a_1TS+a_3S-a_1T^3-a_2T^2-a_6$ in 
$K[T][S]$ we see that is is enough
to consider polynomials with degree in $S$ less than $2$.
In this case the result will follow from the equation
$v(f(x,y))=\min\{v(a)\mid a \text{ is a coefficient of }f\}$, which 
is in turn equivalent to $\pi(f(x,y))\neq 0$ if $\pi(f(T,S))\neq 0$ for a $f\in \Oval[T,S]$
with degree in $S$ less than $2$. To see this last point note 
that $\pi(x)\notin k$ and so it is transcendental over $k$.
So we have $\pi(h)$ is an irreducible polynomial in $k[T,S]$ (it has degree $2$ in $S$,
and the degree in $T$ shows it has no root which is a polynomial in $K[T]$), and so
$\pi(h)(\pi(x),S)$ is an irreducible polynomial in $k(x)[S]$.
If
$\pi(f)(x,y)=0$ then $\pi(h)(x,S)|\pi(f)(x,S)$, and so by degree reasons
$\pi(f)(x,S)=0$, and this implies
$\pi(f)=0$.
\end{proof}

\begin{lema}\label{opa-ell}
Let $K$ be henselian valued field of residue characteristic $0$, 
and assume given a Weierstrass equation with coefficients in $\Oval$. 
Then the map
$E(K)\to \tilde{E}(k)$ is opaque.
\end{lema}

\begin{proof}
By Lemma \ref{type-opa} we have to see that for an elementary extension
$K'$ of $K$ and $(x,y),(x',y')\in E(K')$ such that $\pi(x,y),\pi(x',y')\notin \tilde{E}(k)$ and $\pi(x,y)\equiv_K\pi(x',y')$ then $(x,y)\equiv_K (x',y')$.
We may assume that $\pi(x,y)=\pi(x',y')$.

By elimination of quantifiers relative to the RV sort, we have to see that 
$rv(K(x,y))\equiv_K rv(K(x',y'))$ and that $K(x,y)\cong K(x',y')$ as valued fields.
We have that $K(x,y)\cong K(x',y')$ as valued fields as in the previous proof.
The previous proof also shows that the image of $K(x,y)$ and of $K(x',y')$ under $v$
is $v(K)$, and the kernel of the map $rv(K(x,y)^{\times})\to \Gamma$ and of the
map $rv(K(x',y'))^{\times}\to \Gamma$ is
$k(\pi(x),\pi(y))^{\times}$. We conclude that $rv(K(x,y))=rv(K(x',y'))$
\end{proof}

The following aside is a consequence of the previous discussion.

\begin{prop}\label{st-dom-ell}
Let $K$ be a model of ACVF of residue characteristic not $2$ or $3$.
Let $E$ be an elliptic curve over $K$.
Then $E_0(K)\subset E(K)$
 is the maximal stably dominated connected subgroup of $E(K)$, obtained
in \cite[Corollary 6.19]{metastable}.
\end{prop}

\begin{proof}
From  Proposition \ref{tate-uniformization} we get that $E(K)/E_0(K)$ is $\Gamma$-internal
and in fact
$E(K)/E_0(K)\cong O(v(\Delta))/\mathbb{Z}v(\Delta)$ for $\Delta$ the discriminant of 
minimal Weierstrass equation of $E$.

From Lemma \ref{opa-ell} we get that the type $p(x)$ that says $x\in E_0$ and $\pi(x)\notin \tilde{E}(k)$ is a complete type. By quantifier elimination in the 3-sorted language $\mathcal{L}_{\Gamma k}$ for valued fields (with sorts for the valued field, the value group and the residue field) (see \cite[Theorem 2.1.1(3)]{HHM}) 
we conclude that every definable family $\{S_a\}_a$ with $S_a\subset k^n$ is 
definably piecewise in $a$ 
of the form $S_a=T_{f(a)}$ for a definable function $f$ into $k^m$
and $T_{b}$ a definable family in the pure field language. From this and the fact
that all types in the stable theory ACF are definable we obtain that the type $p$ 
is definable. The type $p$ is $E_0$-invariant, as $\pi_*(gp)=\pi(g)\pi_*(p)=\pi_*(p)$.
The second equation holds because $\tilde{E}_k$ is an irreducible variety over $k$.

Finally the type $p$ is stably dominated via $\pi$ over some parameter set $A$ defining $E$. 
We recall this means that if $b\vDash p|A$ and $\pi(b)\vDash \pi_*(p)|B$ then $f(a)\vDash p|B$
for $A\subset B$. The first condition implies $b\in E_0(K)$ and the second that 
$\pi(b)$ is generic in $\tilde{E}$ over $k_B$ (where $k_B$ is the residue field of $acl(B)$)
\end{proof}

In residue characteristic $2$ and $3$ the equations for a minimal Weierstrass equation
become more complicated and it is not clear to us how to prove for example that 
it stabilizes on
finite extensions. 
This is likely to be the case however, because the construction 
in \cite{metastable} does not have this restriction.

\begin{lema}\label{opa-group}
If $f:G\to H$ is a definable group morphism which is opaque, and $X\subset G$ is
a type-definable subgroup, then either $f(X)$ is finite or $X=f^{-1}f(X)$
\end{lema}

\begin{proof}
Let $K$ be the kernel of $f$.
Suppose $f(X)$ is infinite. 
 There is $C$ a small
set such that if $h\in H\setminus C$ then 
$f^{-1}(h)\cap X=\emptyset$ or $f^{-1}(h)\subset X$.
Indeed if $X$ is a small intersection of definable sets $D_i$ and $S_i$ is a finite
exceptional set of $D_i$ with respect to $f$, then $C=\bigcup_i S_i$ is a small 
exceptional set of $X$ with respect to $f$.

Take $y\in f(X)\setminus C$. Then $K\subset f^{-1}(y)f^{-1}(y)^{-1}\subset X$ as required.
\end{proof}
We will need the following modification in later in Lemma \ref{psfv0-rv}. 

\begin{lema}\label{opa-group2}
Let $G$ be a definable group, and let $\{G_i\}_{i\in I}$ a small downwards
directed net of definable normal subgroups of $G$.
Suppose that for every $Y\subset G$ definable, there exists $i\in I$
such that $Y$ is almost saturated relative to $G\to G/G_i$.
Suppose $X\subset G$ is a type-definable subgroup.
Then either the image of $X$ in $G/\bigcap_i G_i$ is finite
or $X$ contains $\bigcap_i G_i$
\end{lema}

In fact, both this lemma and Lemma \ref{opa-group} can be proved
together by suitably defining opaqueness not only for definable sets but
for pro-definable sets and considering the map
$G\to G/\bigcap_i G_i=\lim_i G/G_i$. In order to keep the definition
simple we prefer to repeat the proof with the necessary modifications.
\begin{proof}
If $Y\subset G$ is definable then it being almost saturated with respect
to $G\to G/G_i$ implies a fortiori that it is almost saturated with 
respect to $f:G\to G/\bigcap_i G_i$. We conclude as in the proof
of Lemma \ref{opa-group} that $X$ has a small exceptional set
$C\subset G/\bigcap_i G_i$ relative to $f$.
Now we can use compactness and saturation to see that if $f(X)$ is infinite
then $f(X)\setminus C$ is non-empty. So if $r\in f(X)\setminus C$
then $\bigcap_i G_i\subset f^{-1}(r)f^{-1}(r)^{-1}\subset X$.
\end{proof}

The map $RV\to \Gamma$ is usually not opaque, but we have the following weakening.
\begin{lema}\label{opases}
Let $0\to A\to B\to C\to 0$ be an exact sequence of abelian groups as in
 Situation \ref{rvsit}.
If $X\subset B$ is definable then there is an $n\in\mathbb{Z}_{>0}$ such that $X$ is almost saturated
with respect to $B\to B/nA$.
\end{lema}

\begin{proof}
This condition is closed under boolean combinations in $X$ so we just need to check the cases
described in Proposition \ref{qeses}. These are clear.
\end{proof}

\subsection{Orthogonality}\label{ortsec}
Suppose that $M$ is a structure in some language $L$.
If $A$ and $B$ are $0$-definable sets in $M$, then $A$ and $B$ are called orthogonal if
every definable subset $C\subset A^n\times B^m$ is a boolean combination of products
of definable subsets of $A^n$ and $B^m$. 

An example of orthogonal sets is given by two sorts $A$ and $B$
 in a sorted language $L$, where
$M$ eliminates quantifiers, and such that there are no functions from $A^n$ to a sort
different from $A$, or from $B^m$ to a sort different from $B$, or from a product of 
at least one copy
of $A$ and of $B$ (in any order)
into any sort. For example the sorts of the residue field and the value group 
in a non-trivially valued algebraically
closed valued field, or in a henselian valued field of residue characteristic 0.

\begin{lema}\label{val-and-res}
If $A$ and $B$ are orthogonal, and $X\subset A^n$ is a definable set and
$f:X\to B^m$ is a definable function, then the image of $f$ is finite.
\end{lema}
This is clear from considering the graph of $f$.

We will also need a similar result for maps between sets definable in $\Gamma$ or $k$
and the main sort in ACVF. These are well known, but we
include a proof for convenience.

The $n=m=1$ case of the 
following lemma can be found as part of \cite[Theorem 2.4.13]{HHM}.
\begin{lema}\label{val-to-home}
Suppose $K$ is a model of ACVF. Suppose $X\subset \Gamma^n$ is definable and $f:X\to K^m$ is definable. Then the image
of $f$ is finite.
\end{lema}

\begin{proof}
We may assume $m=1$. It follows for example from $C$-minimality of ACVF that every infinite definable subset $D$ of $K$ has non-empty interior, so it definably embeds the closed unit ball $\Oval$. If $f:X\to K^m$ has infinite image, we may thus assume $f(X)=\Oval$. But then $\pi\circ f:X\rightarrow k$ is surjective, contradicting Lemma \ref{val-and-res} and the fact that $\Gamma$ and $k$ are orthogonal in ACVF.
\end{proof}

\begin{lema}\label{res-to-home}
Suppose $K$ is a model of ACVF. Suppose $X\subset k^n$ is definable and $f:X\to K^m$ is definable. Then the image of 
$f$ is finite.
\end{lema}

\begin{proof}
The proof is the word by word the same as the previous one, Interchanging the roles of $k$ and $\Gamma$.
\end{proof}

\section{Subgroups of one-dimensional algebraic groups}\label{subgrosec}
Here we give the type-definable subgroups of an algebraic group of dimension $1$
in a model of ACVF or a model of PL$_0$.
\subsection{Subgroups of the additive group}\label{addgrosec}
Here we give the type-definable subgroups of the additive group for 
a model of ACVF and a model of PL$_0$. We give the argument in a more
general form for reference, as it reappears in the analysis of the other 
one dimensional groups.
\begin{defi}
Suppose $G$ is a commutative group and $I_{\infty}$ is a totally ordered
set with a maximal element denoted  $\infty$. Denote $I=I_{\infty}\setminus\{\infty\}$.
A group filtration\footnote{In the literature a group filtration $w$ is sometimes called a \emph{pre-valuation} on $G$, and a \emph{valuation} if $w$ is separated.} of $G$ by $I$ is a map $w:G\to I_{\infty}$ such that
$w(a+b)\geq\text{min}\{w(a),w(b)\}$, $w(-a)=w(a)$ and $w(0)=\infty$.
The filtration is separated if $w(a)=\infty$ implies that $a=0$.

Given a group filtration and $r\in I$ we denote $G_r=\{a\in G\mid w(a)\geq r\}$
and $G_r^-=\{a\in G\mid w(a)>r\}$.

Note that $G_r$ and $G_r^-$ are subgroups of $G$.
\end{defi}

\begin{lema}\label{pll}
Suppose $G$ is a commutative group definable in some theory.
Suppose $w:G\to I_{\infty}$ is a definable separated group filtration.

Assume that
\begin{enumerate}
\item For every $r\in I$, $G_r/G_r^-$ is infinite.
\item For every $r\in I$, 
any type-definable subgroup of $G_r/G_r^-$ equals $0$ or $G_r/G_r^-$.
\item  For every $r\in I$, the map $G_r\to G_r/G_r^-$ is opaque.
\item Every type-definable upwards closed subset of $I$ is a small intersection
of intervals.
\end{enumerate}
Then every type-definable subgroup of $G$ is an intersection of definable
groups, and every definable subgroup is of the form $0$, $G$, $G_r$ or $G_r^-$.
\end{lema}
\begin{proof}
Let $X\subset G$ be a type-definable subgroup.
Suppose $r\in w(X)$, and take $\pi:G_r\to G_r/G_r^-$ the projection. 
Then $\pi(X\cap G_r)$ is a non-zero type-definable subgroup of $G_r/G_r^-$, 
and so
it is equal to $G_r/G_r^-$. 
By Lemma \ref{opa-group} we get $G_r\subset X$.
We conclude $w(X)\subset I_{\infty}$ is a type-definable upwards closed set,
and so it is a small intersection of intervals. 
And so $X$ corresponds to the corresponding intersection of $G_r$ and $G_r^-$.
\end{proof}

\begin{lema}\label{psf-sgr}
Let $k$ be a pseudo-finite field of characteristic $0$,
and $G$ be a connected one dimensional algebraic group over $k$.
If $X\subset G(k)$ is an infinite type-definable subgroup,
then $X=\bigcap_n XG(k)^n$, and $G(k)^n$ is of finite index in $G(k)$ for every $n\geq1$, so $XG(k)^n$
is a definable subgroup of $G$ of finite index.
\end{lema}

\begin{proof}
The theory of $k$ is supersimple, so $X$ is an intersection of definable groups (see \cite[Theorem 5.5.4]{simple}).
We may thus assume that $X$ is an infinite definable subgroup.

The theory of $k$ is geometric and admits elimination of imaginaries after adding
some constants (see \cite{psf-algebraically-bounded} and 
\cite[Proposition 1.13]{psf-imaginaries}), so
in this case we have that $\mathrm{dim}(X)=1$, and so $\mathrm{dim}(G(k)/X)=0$.
We conclude that $G(k)/X$ is finite. But then $G(k)^n\subset X$ for some $n\geq1$ as required.

The fact that $G(k)^n$ is of finite index in $G(k)$ is shown as follows.
The $n$th-power map $G(k)\to G(k)$ is a group map with finite kernel,
because the $k$-rational $n$-torsion points form a subset of the 
$k^{alg}$-rational $n$-torsion points, and these form finite subgroup, by examining each of the possible groups $G$ in the classification
of one dimensional connected algebraic groups.
The $k^{alg}$-rational $n$-torsion points form a subgroup isomorphic to $0$,
$\mathbb{Z}/n\mathbb{Z}$ or $\mathbb{Z}/n\mathbb{Z}\times \mathbb{Z}/n\mathbb{Z}$
in the additive, multiplicative and elliptic case respectively.
From here using that the theory is geometric we conclude that
$\mathrm{dim}(G(k)^n)=1$ and $\mathrm{dim}(G(k)/G(k)^n)=0$, and so 
$G(k)/G(k)^n$ is finite, as required.
\end{proof}

\begin{lema}\label{psf-add}
Let $k$ be a pseudo-finite field of characteristic $0$.
If $X\subset (k,+)$ is a type-definable subgroup,
then $X=k$, or $X=0$.
\end{lema}

\begin{proof}
This follows from Lemma \ref{psf-sgr} because $nk=k$ and $k$ is torsion
free.
\end{proof}

\begin{prop}\label{psfv0-add}
Suppose $K$ is a model of PL$_0$.
Then every non-trivial 
type-definable subgroup $X\subset K$ of $(K,+)$ 
is an intersection of balls around the origin.
\end{prop}

\begin{proof}
Using the filtration $v:K\to \Gamma_{\infty}$, this is a consequence of
Lemma \ref{pll}, which applies by Lemma \ref{psf-add} and \ref{opa+}.
\end{proof}

\begin{lema}\label{acvfl}
Suppose $G$ is a commutative 
group definable in some theory and $v:G\to I_{\infty}$ is a definable group
filtration.
Suppose every definable subset of $G$ is a boolean combination of translates of $G_r$ and $G_r^{-}$
or the trivial subgroups.
Suppose $G_r/G_r^-$ is infinite for every $r\in I$.
Then every type-definable subgroup of $G$ is a small intersection of definable subgroups
and a definable subgroup is trivial or finite or a finite extension of a $G_r$ or $G_r^{-}$
\end{lema}

\begin{proof}
Let us call $G_r$ or $G_r^-$ closed and open balls respectively.

In this situation we claim that if $S$ is a nonempty 
set of the form $C\setminus \bigcup_{i=1}^m a_i+C_i$
where each $C, C_i$ are trivial groups or balls, then
$C=S-S$.
Assume first $C=G_r^-$ is an open ball, but not a closed ball. 
In this case $S$ contains a nonempty set of the form
$G_r^{-}\setminus G_s$ for $s>r$.
So without loss of generality $S=G_r^-\setminus G_s$. For $x\in G_r^-$, as $G_r^-\supsetneq G_{\mathrm{min}(v(x),s)}$ there is $y\in S$ such that $v(y)<v(x)$. Then $x+y,y\in S$ and thus $x\in S-S$.

Now suppose $S=G_r\setminus X$. 
Then without loss of generality $X$ is a union of a finite
number of translates of $G_r^{-}$, say $X=F+G_r^-$ with $F$ finite.
Denote $\pi:G_r\to G_r/G_r^{-}$ and let $x\in G_r$ be arbitrary. Choose $y\in G_r$ such that $\pi(y)$ is non-algebraic over $\pi(F)\cup\{\pi(x)\}$. Then $\pi(x+y),\pi(y)\not\in\pi(F)$, in other words $x+y,y\in S$ and thus $x\in S-S$.

The case of $S$ equals $G\setminus X$ where
$X$ is a union of translates of groups of the form $G_s$, $G_s^-$ or $0$, has a very similar proof, which depending on whether $I$ has a minimal element or not is as the proof for $G_r$, or the proof for $G_r^-$, respectively.

This finishes the proof
of the claim.

Let $X\subset G$ be a type-definable subgroup and let $X\subset D$ be a definable set.
We have to find a definable $B$ such that $X\subset B\subset D$ and $B$ is a subgroup
of the form claimed.
Let $S$ be a symmetric definable set with $X\subset S$ and $3S\subset D$.
Then $S$ is a finite union of the form 
$S=\bigcup_k (a_i+E_i)$ where each $E_i$ is of the form discussed in the first 
paragraph. So there is $B_i$ which is either a ball or trivial such that
$B_i=E_i-E_i$ and $E_i\subset B_i$.
If $B=\bigcup_iB_i$ then $B\subset S-S$ and $S\subset F+B$ for $F$ finite.

Now if $H$ is the group generated by $F+B$ then $H/B$ is a finitely generated group, so
$H=B_t\oplus \mathbb{Z}^s$ where $B_t/B$ is the torsion part of $H/B$ and so 
is a finite group.
As $X+B$ is a type-definable subgroup of $H$ one concludes by compactness that 
$X+B\subset B_t$. And so being the inverse image of a subgroup of the finite group
$B_t/B$ the subgroup $X+B$ is definable. So $X\subset X+B\subset D$ is the group we wanted.
\end{proof}

\begin{prop}\label{acvf-add}
If $K$ is a model of ACVF and $X\subset K$ is a type-definable subgroup of the additive group,
then it is an intersection of definable groups, and any definable group is a finite extension
of a ball around $0$.
\end{prop}

\begin{proof}
This is a particular case of Lemma \ref{acvfl}.
\end{proof}

\subsection{Subgroups of the multiplicative group}\label{mulgrosec}
\begin{lema}\label{sesfin}
Let $0\to A\to B\to C\to 0$ be an exact sequence of abelian groups as in \ref{rvsit}.
Take $n\in \mathbb{Z}_{>0}$.
Then the sequence $0\to A/nA\to B/nA\to C\to 0$ is split exact and any splitting is 
definable.
\end{lema}
For a proof see for example \cite[Proposition 8.3]{one-dimensional-p-adic}. 
\begin{lema}\label{ses-grp}
Let $0\to A\to B\to C\to 0$ be an exact sequence of abelian groups as in \ref{rvsit}, and let $X\subset B$ be a type-definable subgroup. Then either $X\subset A$ or 
$X=\bigcap_n X_n$ where each $X_n$ is a finite extension of a group of the form
$g^{-1}(Y_n)\cap nB$ for $Y_n\subset C$ a type-definable subgroup.
\end{lema}

\begin{proof}
Let us denote $g:B\to C$.
As $C$ is torsion-free $g(X)$ is either trivial or infinite.
If $g(X)$ is trivial, then $X\subset A$. Otherwise
by Lemma \ref{opa-group2}, which applies by Lemma \ref{opases},
we have that $\bigcap_n nA\subset X$. 
We take $X_n=X+nA$, so that $nA\subset X_n$, and $X=\bigcap_n X_n$ by compactness.
If we take $Z_n=nB\cap X_n$ then $X_n/Z_n\cong (X_n+nB)/nB$ is finite, because
$B/nB$ is.

We claim that $Z_n\subset g^{-1}g(Z_n)\cap nB\subset X_n$. This follows from
 Lemma \ref{sesfin}. Indeed, denote $W=g^{-1}g(Z_n)\cap nB$, and note that
$nA\subset W$ and $nA\subset X_n$. Take $s:C\to B/nA$
a section of $B/nA\to C$, so that $B/nA=A/nA\oplus s(C)$. Note that
$nB/nA= ns(C)$. It follows that $Z_n/nA=X_n/nA\cap nB/nA=ns(C)\cap X_n/nA$.
So we have
$g^{-1}g(Z_n)/nA= A/A_n+Z_n/nA=A/nA\oplus (ns(C)\cap X_n/nA)$. Finally, we have
$W/nA=g^{-1}g(Z_n)/nA\cap nB/nA =(A/nA\oplus (ns(C)\cap X_n/nA))\cap ns(C)=
ns(C)\cap X_n/nA\subset X_n/nA$ as claimed.

We conclude that the lemma holds with $Y_n=g(Z_n)$.
\end{proof}

\begin{lema}\label{z-add}
Let $\Gamma$ be a $Z$-group. If $X\subset \Gamma$ is a type-definable subgroup of $(\Gamma,+)$ 
then $X$ is a small intersection of 
groups of the form $o(a)$ and $n\Gamma$ for $n\in\mathbb{Z}_{>0}$.
\end{lema}
For a proof see \cite[Proposition 3.5]{one-dimensional-p-adic}.

Now we put together Lemmas \ref{z-add}, \ref{psf-sgr} and \ref{ses-grp}.

\begin{lema}\label{psfv0-rv}
Suppose $K$ is a model of PL$_0$.
If $X\subset RV$ is a type-definable subgroup, then it is a small intersections of groups
which are finite extensions of groups of one of the following forms:
\begin{enumerate}
\item the trivial group
\item $(k^{\times})^{n}$
\item $RV^n\cap v^{-1}(o(a))$
\item $RV^n$
\end{enumerate}
\end{lema}

\begin{prop}\label{psfv0-mult}
Suppose $K$ is a model of PL$_0$.
Then every 
type-definable subgroup $X\subset K^{\times}$ is a small intersection of groups
which are finite extensions of groups of one of the following forms:
\begin{enumerate}
\item $(\Oval^{\times})^{n}$
\item $o(a)\cap (K^{\times})^{n}$
\item $(K^{\times})^{n}$
\item $U_a$ for $a>0$.
\end{enumerate}
\end{prop}

\begin{proof}
By Lemma \ref{rv-is-opa}, Lemma \ref{psfv0-rv} 
and Lemma \ref{opa-group}, if $rv(X)$ is infinite then we have cases
1-3.
Otherwise we may take $X\subset U_1$ and then we conclude using Lemma \ref{pll}
with the filtration $w:U_1\to \Gamma_{\infty}$ given by $w(x)=v(1-x)$.
This lemma applies by Lemma \ref{psf-add} and Lemma \ref{mult-opa}.
\end{proof}

Next we do the analysis for ACVF.

\begin{lema}\label{q-add}
If $\Gamma$ 
is a divisible ordered abelian group, and $X\subset \Gamma$ is a type-definable subgroup
then $X$ is a small intersection of groups of the form $o(a)$ or it is trivial.
\end{lema}

The proof in \cite[Proposition 3.5]{one-dimensional-p-adic} works.

The following lemma is well known, but we add a simple, self-contained proof
for the benefit of the reader.
\begin{lema}\label{acf-sgr}
Suppose $k$ is an algebraically closed field.
Let $G$ be a one-dimensional, connected, algebraic group over $k$.
Let $X\subset G(k)$ be an infinite type-definable subgroup.
Then $X=G(k)$.
\end{lema}
\begin{proof}
Let $Y\subset G(k)$ a definable set that contains $X$.
We have to see $Y=G(k)$.
Let $Z$ be a definable set with $X\subset Z$ and $Z-Z\subset Y$.

It is a basic fact that a definable subset of an irreducible, one-dimensional,
algebraic variety is finite or cofinite. So we conclude that $Z$ is cofinite
in $G(k)$, say $G(k)\setminus Z=F$ is finite.
If we take $x\in G(k)$ and $z\in G(K)\setminus(F \cup (x-F))$,
then $x=(x-z)+z$ belongs to $Z-Z$. As $x$ is arbitrary we conclude
$G(k)=Z-Z=Y$ as required.
\end{proof}
\begin{lema}\label{valuation-formal-group}
If $K$ is a model of ACVF,
 and $G$ is a definable commutative group in $K$,
and $f:\Mval\to G$ is a bijection satisfying $f^{-1}(f(x)+f(y))=x+y+r$ for an 
$r$ with $v(r)>{\rm Max}\{v(x),v(y)\}$,
$f(0)=0$ and $f^{-1}(-f(x))=-x+s$ for an $s$ satisfying $v(s)\geq v(x)$,
then the funtion $w:G\to \Gamma_{\infty}$ given by $w=vf^{-1}$ satisfies the hypotheses of Lemma \ref{acvfl}.
\end{lema}

\begin{proof}
Suppose $x,y\in G$ such that $w(x),w(y)\geq t$. We have to see that $w(x+y)\geq t$.
There are $a,b\in \Mval$ such that $x=f(a)$ and $y=f(b)$. Then $f^{-1}(x+y)=a+b+r$ with
$v(r)>{\rm Max}\{v(a),v(b)\}\geq t$ and so $w(x+y)=v(a+b+r)\geq t$.

Now suppose that $w(x)=t$, we will see that $w(-x)\geq t$. If $x=f(a)$ then
$f^{-1}(-x)=-a+s$, with $v(s)\geq v(-a)=t$. From here $w(-x)=v(-a+s)\geq t$.
Applying this to $-x$ we get $w(x)=w(-x)$

If we call $G_t$ and $G_t^-$ the sets given by $w(x)\geq t$ and $w(x)>t$ in $G$,
then we show that $f(a+B_t)=f(a)+G_t$ and $f(a+B_t^-)=f(a)+G_t^-$.
Indeed if $x\in f(a)+G_t$ then there is $b\in B_t$ such that $x=f(a)+f(b)$.
Then $f^{-1}(x)=a+b+r$. We have $v(r)>v(b)$ and so $v(b+r)=v(b)\geq t$, so
$x\in f(a+B_t)$.
In the other direction if $x\in f(a+B_t)$ then there is $b\in \Mval$ such that
$x=f(a)+f(b)$. In this case $f^{-1}(x)=a+b+r$ with $v(r)>v(b)$ and so $v(b)=v(b+r)\geq t$.
From this $wf(b)=v(b)\geq t$, and we conclude $f(b)\in G_t$ and $x\in f(a)+G_t$.
The proof with $B_t^-$ and $G_t^-$ is similar.

Now Swiss cheese decomposition applied to $\Mval$ implies that every definable subset of $G$
is a boolean combination of translate of the trivial group or $G_r, G_r^{-}$ as required.

Note that $f$ produces a group isomorphism $B_t/B_t^-\cong G_t/G_t^-$,
so $G_t/G_t^-$ is infinite, because it is isomorphic to $k$.
\end{proof}

In fact in the applications one has $v(r)\geq v(x)+v(y)$ and $v(s)\geq 2v(x)$.

\begin{prop}\label{acvf-mult}
Suppose $K$ is an algebraically closed field. Suppose
$X\subset K^{\times}$ is a type-definable subgroup.
Then $X$ is a small intersection of groups of the form $o(a)$, or it is equal to $K^{\times}$ or it is equal to $\Oval^{\times}$, or $X$ is a
a small intersection of definable groups, which are finite extensions
of groups of the form $U_r$ or $U_r^{-}$.
\end{prop}

\begin{proof}
If $v(X)$ is not trivial then it is infinite because $\Gamma$ is torsion free,
so $X=v^{-1}v(X)$ by Lemma \ref{opa-group} and Lemma \ref{v-opa},
so $X$ is either $K^{\times}$ or a small intersection of groups
of the form $o(a)$ by Lemma \ref{q-add}.

Otherwise $X\subset \Oval^{\times}$. If  $\pi:\Oval^{\times}\to k^{\times}$ is the residue map
then $\pi(X)\subset k^{\times}$ is a type-definable group. We conclude
that $\pi(X)=k^{\times}$ or it is finite, by Lemma \ref{acf-sgr}.
In the first case because $\pi$ is opaque by Lemma \ref{mult-opa}, we conclude $X=\Oval^{\times}$, by Lemma \ref{opa-group}.
In the other case we may assume $X\subset 1+\Mval$.

For this group we have the function $w:1+\Mval\to \Gamma_{\infty}$ given by
$w(1+x)=v(x)$, this satisfies the properties of Lemma \ref{acvfl},
 see Lemma \ref{valuation-formal-group}.
\end{proof}

\subsection{Subgroups of the elliptic curves}\label{ellgrosec}

\begin{prop}
Let $K$ be a model of PL$_0$. Let $E$ be an elliptic curve over $K$ given by a given Weierstrass equation with 
coefficients in $\Oval$. Then if $X\subset E_0(K)$ is a type-definable subgroup then $X$ is a small intersection
of groups which are finite extensions of groups of one of the following forms:
\begin{enumerate}
\item $nE_0(K)$
\item $E_r$
\item $E_r^{-}$
\end{enumerate}
\end{prop}
\begin{proof}
The proof is similar to the one in 
 Proposition \ref{psfv0-mult}. 

In detail, if we denote $\pi:E_0(K)\to \tilde{E}_0(k)$ the reduction map,
then if $\pi(X)$ is infinite we conclude we have the first case by
 Lemma \ref{opa-ell},
Lemma \ref{opa-group},
Lemma \ref{psf-sgr}
 and 
Lemma \ref{multiplication-by-n-henselian}.

Otherwise we may take $X\subset E_0^-$ 
and then we conclude using Lemma \ref{pll}
with the filtration $w:E_0^-\to \Gamma_{\infty}$ given by 
$w(x,y)=v(-\frac{x}{y})$ as described in Lemma \ref{filtration-elliptic}.
Lemma \ref{pll} applies because of Lemma \ref{filtration-elliptic-subquotient},
Lemma \ref{psf-add} and Lemma \ref{opa-ell2}.
\end{proof}

\begin{prop}\label{acvf-ell}
Let $K$ be a model of ACVF.
 Let $E$ be an elliptic curve over $K$ given by a given Weierstrass equation with 
coefficients in $\Oval$.

If $X\subset E_0$ is a type-definable subgroup then it is equal to $E_0$ 
or it is a small intersection of groups which are finite
extensions of groups of the form
$E_r$ or $E_r^{-}$.
\end{prop}
\begin{proof}
The proof is similar to the one in Proposition \ref{acvf-mult}.

In detail,
denote $\pi:E_0(K)\to \tilde{E}_0(k)$ the reduction map.
then $\pi(X)\subset \tilde{E}_0(k)$ is a type-definable group. 
We conclude $\pi(X)=\tilde{E}_0(k)$ or it is finite, by
Lemma \ref{acf-sgr}.
In the first case because $\pi$ is opaque by Lemma \ref{opa-ell}, 
we conclude $X=E_0$ by Lemma \ref{opa-group}.
In the other case we may assume $X\subset E_0^-$.
For this group we have the function $w:E_0^-\to \Gamma_{\infty}$ given by
$w(x,y)=v(-\frac{x}{y})$, which 
satisfies the properties of Lemma \ref{acvfl}, see 
Lemma \ref{valuation-formal-group} and 
Lemma \ref{filtration-elliptic}. 
\end{proof}

\subsection{Subgroups of the twisted multiplicative groups}\label{twimulgrosec}
In this section we handle the remaining case for $K$ a model of PL$_0$, i.e. for a pseudo-local  valued 
field of residue characteristic $0$.

Assume that $K$ is a Henselian valued field of residue characteristic not $2$.
And assume first that $d\in K^{\times}\setminus (K^{\times})^{2}$ is such that $v(d)=0$.
Since $K$ is Henselian there exists a unique extension of the valuation on 
$K$ to $L=K(\sqrt{d})$ (we will also see this directly). This valuation is then
$v:L\to \mathbb{Q}\otimes_\mathbb{Z}\Gamma_{\infty}$ given by the equation
$v(\alpha)=\frac{1}{2}v(N(\alpha))$.
If $\alpha=a+b\sqrt{d}$ and $a,b\in \Oval$ but $a$ and $b$ are not both in $\Mval$,
then $v(\alpha)\geq 0$ and if $v(\alpha)\geq r>0$ then 
$v(a^2-b^2d)\geq 2r$ and $v(a^2)=v(b^2d)=0$. This implies then that the equation
$x^2=d$ has approximate solution $ab^{-1}$ so an application of Hensel's Lemma implies
$d\in (\Oval^{\times})^{2}$. We conclude that $v(a+b\sqrt{d})=\text{min}\{v(a),v(b)\}$.
We have the short exact sequences $1\to \Oval_L^{\times}\to L^{\times}\to \Gamma_L\to 0$,
$1\to U_L^{-}\to \Oval_L^{\times}\to k_L^{\times}\to 1$ and
the bijection $U_L^{-}\to \Mval_L$ that gives a filtration
$U_{L,r}$ and $U_{L,r}^-$ satisfying $U_{L,r}/U_{L,r}^-\cong (L,+)$. 
We now describe the restrictions of this on $G(d)$.
To start notice that  $G(d)\subset \Oval_L^{\times}$.
Next, we have seen that $k_L=k(\sqrt{\bar{d}})$, and an application of Hensel's Lemma
shows that the image of $G(d)$ in $k_L$ is $G(\bar{d})$.
Denote the kernel of $G(d)\to G(\bar{d})$ by $G(d)^-$. 
It consists of the elements of the form $a+b\sqrt{d}$ such that $v(a-1),v(b)>0$ and 
$a^2-b^2d=1$. In this case $v(a^2-1)=v(a-1)=v(b^2d)=2v(b)$, so in particular
$v(\alpha-1)=v(b)$.
We conclude now that the map $f:G(d)^-\to \Mval$ given by $a+b\sqrt{d}\mapsto b$ is a bijection,
(that it is surjective follows from Hensel's Lemma) and satisfies that
$f(\alpha_1\alpha_2)=f(\alpha_1)+f(\alpha_2)+r$ where $r>vf(\alpha_1)+vf(\alpha_2)$
and $f(\alpha_1^{-1})=-f(\alpha_1)$.
In particular the groups $G(d)_{r}=U_{L,r}\cap G(d)$ and $G(d)_{r}^-=U_{L,r}^-\cap G(d)$ 
satisfy
$G(d)_r/G(d)_r^-=B_r/B_r^-\cong (k,+)$ via $f$.

\begin{lema}\label{opa-tmult}
Let $K$ be a Henselian valued field residue characteristic $0$.
Let $d\in K^{\times}\setminus (K^{\times})^{2}$ be such that $v(d)=0$.
Then the map $G(d)\to G(\bar{d})$ is opaque.
\end{lema}

\begin{proof}
This proof is the same as in Lemma \ref{opa-ell}.
\end{proof}

\begin{lema}\label{opa2-tmult}
Let $K$ be a Henselian valued field of residue characteristic $0$.
Let $d\in K^{\times}\setminus (K^{\times})^{2}$ be such that $v(d)=0$.
Then the maps $G(d)_r\to G(d)_r/G(d)_r^-$ are opaque.
\end{lema}

\begin{proof}
Under the bijection $G(d)^-\cong \Mval$ and after rescaling this becomes Lemma \ref{opa+}.
\end{proof}
Now assume $K$ is Henselian of residue characteristic not $2$ and assume the value
group is a $Z$-group. Take $d\in K$ such that $v(d)=1$.
As before we know that the valuation on $K$ extends in a unique way to a
valuation on $K(\sqrt{d})$. In this case we have
$v(a+b\sqrt{d})=\text{min}\{v(a),v(b)+\frac{1}{2}\}$, $\Oval_L=\Oval+\Oval\sqrt{d}$ and $k_L=k$, 
$\Gamma_L=\frac{1}{2}\Gamma$.
As before $G(d)\subset \Oval_L^{\times}$. Now the residue map $G(d)\to k$ takes $G(d)$ onto 
$\{\pm 1\}$. Its kernel is denoted $G(d)^-$ and it consists of the elements
$a+b\sqrt{d}$ such that $v(a-1)>0$ and $v(b)\geq 0$.
In this case we have the groups $U_{r,L}$ for $r\in \Gamma$ given by the elements
$a+b\sqrt{d}$ such that $v(a-1)\geq r$ and $v(b)\geq r$, and the groups
$U_{r+\frac{1}{2},L}$ consisting of elements $a+b\sqrt{d}$ such that 
$v(a-1)\geq r+1$ and $v(b)\geq r$.
We obtain then the group isomorphisms
$U_{r,L}/U_{r+\frac{1}{2},L}\to B_r/B_{r+1}$ given by $a+b\sqrt{d}\mapsto \overline{a-1}$
and $U_{r+\frac{1}{2},L}/U_{r+1,L}\to B_r/B_{r+1}$ given by 
$a+b\sqrt{d}\mapsto \overline{b}$. If we denote $G(d)_r=U_{r,L}\cap G(d)$ then we have that
$G(d)_r=U_{r+\frac{1}{2},L}\cap G(d)$.
Indeed if $a^2-b^2d=1$ and $v(a-1)>0$, then 
$v(a-1)=v(a^2-1)=v(b^2d)=2v(b)+1$ so if $v(b)\geq 0$ then 
$v(a+b\sqrt{d}-1)=v(b)+\frac{1}{2}$.

\begin{lema}\label{opa3-tmult}
If $K$ is a henselian valued field of residue characteristic $0$, and valued group
a $Z$-group, and if $d\in K$ is such that $v(d)=1$, then the maps
$G_{r}(d)\to B_r/B_{r+1}$ given by $a+b\sqrt{d}\mapsto \bar{b}$ are opaque.
\end{lema}

\begin{proof}
Indeed the map $G(d)^-\to \Oval$ given by $a+b\sqrt{d}\mapsto b$, is bijective by 
Hensel's Lemma, under this bijection the map becomes as in Lemma \ref{opa+}.
\end{proof}

The map $f:G(d)^-\to \Oval$ mentioned in the previous proof satisfied that
$vf(\alpha)=v(\alpha)-\frac{1}{2}$ and $f(\alpha^{-1})=-f(\alpha)$,
$f(\alpha_1+\alpha_2)=f(\alpha_1)+f(\alpha_2)+r$ with $v(r)>v(f(\alpha_1))+v(f(\alpha_2))$.

\begin{prop}\label{psfv0-tmult}
Suppose $K$ is a model of PL$_0$. 
Suppose $d\in K^{\times}\setminus (K^{\times})^{2}$ and let
$G(d)$ be the corresponding twisted multiplicative group.
If $X\subset G(d)$ is a type-definable subgroup then $X$ is a small intersection of groups
which are finite extensions of groups of the form $G(d)^n$ or 
$G(d)_r$ for $r\in \Gamma_{>0}$ or $G(d)^-$.
\end{prop}
\begin{proof}
If $v(d)=0$ this is a consequence of 
Lemmas \ref{opa-tmult}, \ref{psf-sgr},
\ref{opa-group} and 
\ref{pll}. Lemma \ref{pll} applies by Lemma \ref{psf-add}
and Lemma \ref{opa2-tmult}

If $v(d)=1$ this is a consequence of Lemma \ref{pll},
which applies by Lemma \ref{psf-add} and Lemma \ref{opa3-tmult}.
\end{proof}

\section{One-dimensional definable groups}
Here we give a list of all commutative definable groups of dimension $1$ in 
a model of ACVF of residue characteristic not $2$ or $3$ and in a 
model of PL$_0$.
\subsection{Classification of one-dimensional definable groups}\label{maisec}
\begin{lema}\label{finite-kernel}
Let $0\to C\to B\to A\to 0$ be a short exact sequence of abelian groups definable
in some language.
Assume that $C$ is finite, 
and that for $n$ equal to the exponent of $C$, 
 $A/nA$ is finite and the $n$-torsion of $A$ is finite.
Then there is a definable group map $A\to B$ with finite kernel and cokernel.
\end{lema}
\begin{proof}
Denote $f:B\to A$ and $i:C\to B$.
Consider $B'$ the pullback of $B\to A$ via the multiplication by $n$ map
$A\to A$. By definition $B'$ consists of the tuples $(b,a)$ such that $b\in B$,
$a\in A$ and $f(b)=na$.
Note that the canonical map $B'\to B$ has finite kernel and cokernel, in fact
the kernel is isomorphic to the $n$-torsion of $A$, and the cokernel to 
$A/nA$.

We have a short exact sequence 
$0\to C\to B'\to A\to 0$. We claim that the exact sequence
$C\to B'/nB'\to A/nA\to 0$ is a split short exact sequence.
Indeed because $A/nA$ is finite it is enough to see that $C\to B'/nB'$ is injective
with pure image. Because the exponent of $C$ is $n$ and that of $B'/nB'$ divides
$n$ it is enough to see that if $c\in C$ is such that $i(c)\in mB'$ then 
$c\in mC$ for $m$ that divides $n$. So suppose that $(b,a)\in B'$ such that
$m(b,a)=i(c)$. In other words $ma=0$ and $mb=i(c)$. We have that $f(b)=na$, 
but $na=0$ because $ma=0$ and $m$ divides $n$, so $f(b)=0$ and $b=i(c')$ for a $c'\in C$.
We conclude $c\in mC$ as required.

The composition $B'\to B'/nB'$ with a projection $B'/nB'\to C$ is a retraction
of $C\to B'$ so the sequence $0\to C\to B'\to A\to 0$ is definably split exact.
Then the composition of a section $A\to B'$ with the canonical map $B'\to B$
is as required.
\end{proof}

\begin{prop}\label{one-dimensional-acvf}
Let $K$ be a model of ACVF of residue characteristic not $2$ or $3$.

If $G$ is an abelian one-dimensional group definable in $K$, then there is a finite index
subgroup
$H\subset G$  and a finite subgroup $L\subset H$ 
such that $H/L$ is isomorphic to one of the following:

\begin{enumerate}
\item $(K,+)$
\item $(\Oval,+)$
\item $(\Mval,+)$
\item $(K^{\times},\times)$
\item $(\Oval^{\times},\times)$
\item $O(b)/\langle b\rangle$. 
\item $(1+\Mval,\times)$.
\item $(U_r,\times)$.
\item $(U_r^{-},\times)$. 
\item $E_{0}$ for an elliptic curve $E$. 
\item $E_{0}^{-}$ for an elliptic curve $E$.
\item $E_{r}$ for $r> 0$, an elliptic curve $E$.
\item $E_{r}^{-}$ for $r>0$ and an elliptic curve $E$.
\item $O_E(b)/\langle b\rangle$ for an elliptic curve $E$ of multiplicative reduction,
\end{enumerate}
In residue characteristic $0$ the finite group $L$ is not necessary.
\end{prop}
\begin{proof}
The proof in \cite[Proposition 8.2]{one-dimensional-p-adic} goes through.

That the finite kernel is not necessary in residue characteristic $0$ is a consequence
of Lemma \ref{finite-kernel}. Note that it applies to all the groups mentioned.
So now we just need to describe the groups of the form $B/A$ for $B$ a group
in the list and $A$ a finite subgroup of $B$. Cases 1, 2 and 3 are torsion free.
In cases 4 and 5 the group $A$ is finite cyclic and the map $x\mapsto x^n$
produces an isomorphism $B/A\cong B$.
In case 6 a finite subgroup is of the form $\langle c,\eta\rangle$ for 
$\eta\in \Oval^{\times}$ an $m$th root of unity and $c^n=b$. In this case
$B/A=O(c)/\langle \eta, c\rangle$. Now the map $x\mapsto x^m$,
 $O(c)\to O(c)$ produces an isomorphism
$O(c)/\langle \eta \rangle\cong O(c)$, so $B/A$ is as required, as $B/A\cong O(c)/\langle c^m\rangle=O(c^m)/\langle c^m\rangle$.
Cases 7, 8 and 9 are torsion free.
In case 10, if $A\subset E$ is a finite subgroup of an elliptic curve then
$E(K)/A\cong E'(K)$ for an elliptic curve $E'$.
If $E$ has multiplicative reduction then $E_0\cong (\Oval^{\times},\times)$ in the analytic
language, via the Tate map, so any finite subgroup is cyclic and 
$E_0(K)/A\cong E_0(K)$ via the multiplication by $n$ map.
Cases 11, 12, and 13 are torsion free.
Case 14 is as in case 5.
\end{proof}
We also have that the finite kernel is unnecessary in cases 4, 5, 6, 10, 11 and 14
for any characteristic and additionally in case 1 in the mixed characteristic case.
But for example in the mixed characteristic case $\Oval/p\Oval$ surjects onto $k$ and so $\Oval/p\Oval$ is not
finite and Lemma \ref{finite-kernel} does not apply.

We mention that Lemma \ref{finite-kernel} also shows that the finite kernel in 
\cite[Proposition 8.2]{one-dimensional-p-adic} is not necessary.

\begin{prop}\label{one-dimensional-psfv0}
Let $K$ a model of PL$_0$.

If $G$ is an abelian one-dimensional group definable in $K$, then there is
a finite index definable subgroup
$H\subset G$ such that 
$H$ is isomorphic as a definable group to one of the following:

\begin{enumerate}
\item $(K,+)$
\item $(\Oval,+)$
\item $((K^{\times})^n,\times)$
\item $((\Oval^{\times})^n,\times)$
\item $O(b)^n/\langle b^n\rangle $. 
\item $(U_r,\times)$
\item $G(d)^n$ for a $d\in K^{\times}\setminus (K^{\times})^2$ with $v(d)=0$.
\item $G(d)^-$ for a $d\in K^{\times}$ with $v(d)=1$.
\item $G(d)_r$ for a $d\in K^{\times}\setminus (K^{\times})^2$ with $v(d)=0$ or $v(d)=1$
and $r\in \Gamma_{>0}$.
\item $nE(K)/\langle \eta\rangle$ for an elliptic curve $E$ with good reduction 
and a torsion element $\eta$.
\item $nE_{0}(K)$ for an elliptic curve $E$ with multiplicative or additive reduction.
\item $E_{r}$ for $r> 0$ and an elliptic curve $E$.
\item $(O_E(b))^n/\langle b^n\rangle$ for an elliptic curve $E$ of 
split multiplicative reduction.
\end{enumerate}
\end{prop}
\begin{proof}
The proof in \cite[Proposition 8.2]{one-dimensional-p-adic} goes through.

As before the finite kernel is unnecessary by Lemma \ref{finite-kernel}.
We just verify cases 10 and 11. If $E$ has additive reduction then there is a 
short exact sequence $0\to E_1(K)\to E_0(K)\to (k,+)\to 0$, so $E_0(K)$ is torsion free.
In case $E$ has multiplicative reduction there is a short exact sequence
$0\to E_1(K)\to E_0(K)\to A(k)\to 1$ where $A$ is a possibly twisted multiplicative
group over $k$. As $E_1(K)$ is torsion free we conclude that any finite subgroup
of $E_0(K)$ is cyclic, because this is the case in $A(k)$, and also the restriction
of the reduction map is an isomorphism. We conclude that the multiplication by $n$
map $mE_0(K)\to nmE_0(K)$ produces an isomorphism $mE_0(K)/A\cong nm E_0(K)$ where
$A$ is a finite subgroup.
In case the curve has good reduction then the $m'$-torsion is included in 
$\mathbb{Z}/m'\mathbb{Z}^2$ so a 
finite group is cyclic or a direct product of two cyclic groups
$\langle a\rangle\times \langle b\rangle$ where $\text{ord}(a)|\text{ord}(b)$.
In the second case the multiplication by $m=\text{ord}(a)$ map produces an isomorphism
$nE(K)/A\cong nmE(K)/\langle mb\rangle$. 
\end{proof}

\subsection{Definable isomorphisms between one-dimensional definable groups in ACVF$_{(0,0)}$}\label{isogrosec}
Suppose $K$ is a model of ACVF$_{(0,0)}$. In this section we describe to what extent
the list in the main result, Proposition \ref{one-dimensional-acvf}, is redundant. 
This is done in Proposition \ref{acvf-iso}.
In order to make the question precise
define a category consisting of groups definable in $K$, and the maps
between two objects $G$ and $H$ consists of definable group maps
$G_1\to H$ with $G_1$ a subgroup of finite index of $G$, modulo the equivalence
relation that identifies two such maps $G_1\to H$ and $G_2\to H$ if there
is a group of finite index $G_3\subset G_1\cap G_2$ so that the maps coincide
in $G_3$. It is straightforward to see that $G$ and $H$ are isomorphic in
this category if and only if $G$ and $H$ have finite index definable subgroups
isomorphic as definable groups. As all the groups mentioned in the 
main theorem
have no proper definable subgroups of finite index, two such groups
are isomorphic in this category if and only if they are isomorphic as definable
groups.

We call a group equal to one of the list in Proposition \ref{one-dimensional-acvf}
a basic group. The cases $K, \Mval, \Oval$ are called basic of additive type.
The cases $K^{\times},\Oval^{\times}, H(b), U_r^-$ and $U_r$ are called 
basic of multiplicative type.
The cases $E_0$, $E_r^-$, $E_r$ for $r>0$ for an elliptic curve $E$ and 
the case $H_E(b)$ for $E$ of multiplicative reduction, are called basic of elliptic 
type $E$.

\begin{lema}\label{def-to-alg-mor}
Let $K$ be a perfect field with extra structure which is algebraically bounded.
Let $G$ and $H$ be algebraic groups over $K$.
Suppose $A\leq G(K)$ is a definable subgroup of the same dimension
as $G$ and $f:A\to H(K)$ is a definable
group morphism. Then there exists an algebraic group $T$ over $K$ of dimension
equal to the dimension of $G$ and morphisms of algebraic groups over $K$
$p:T\to G$ and $g:T\to H$ such that $p$ has finite kernel and 
for every $x\in A$ there is $y\in T(K)$ such that $p(y)=x$ and $g(y)=f(x)$.
If $B\subset T(K)$ is a definable subgroup of $T(K)$ with no definable finite
index subgroups such that $p(B)\subset A$,
then $fp=g$ in $B$.
\end{lema}

\begin{proof}
Consider $\Gamma\subset G(K)\times H(K)$ the graph of $f$.
It is in definable bijection with $A$ so it has dimension $n$.
Define $T$ the reduced closed sub-scheme of $G\times H$ 
corresponding to the Zariski closure of $\Gamma$ in $G\times H$.
Then $T$ is an algebraic subgroup of $G\times H$ of dimension $n$, by 
Lemmas \ref{closure-group} and \ref{dim-alg}.
The composition of $T\to G\times H$ with the first projection 
$G\times H\to G$ is the group morphism $p$. The image of $p$ contains
$A$ and so it is of dimension $n$. So the kernel of $p$ is finite.
The morphism $g$ is the composition of $T\to G\times H$ with the second
projection $G\times H\to H$.
For the final assertion note that for every $y\in B$, there is $y'\in T(K)$
such that $p(y)=p(y')$ (so that $y'=yk$ for a $k\in \text{ker}(p)$),
 and $g(y')=f(p(y))$.
 If $C$ is the set of $y\in B$
such that $g(y)=f(p(y))$, then the map $C\backslash B\to g(\text{ker}(p))$
defined as $Cy\mapsto k$ if $g(y)k=fp(y)$ is an injective and well defined
map, so we conclude $C$ has finite index in $B$ and $C=B$ as desired.
\end{proof}

\begin{lema}\label{add-mor}
Let $K$ be a model of ACVF$_{(0,0)}$.
The definable group morphisms $\Oval\to K$, $\Mval\to K$ and $K\to K$ 
(even in the analytic language)
are multiplication
by a constant from $K$.
\end{lema}
\begin{proof}
Let $f:\Oval\to K$ be a definable morphism. If $f(1)=k$ then the set of $x\in \Oval$
such that $f(x)=kx$ is a definable group containing 1, so it must be equal to $\Oval$, by
Proposition \ref{acvf-add} (and C-minimality in the analytic case, see Proposition \ref{c-minimal}).

The other cases follow from this because $\Mval=\bigcup_{r>0} B_r$ and
 $K=\bigcup_r B_r$ where each
$B_r$ is isomorphic to $\Oval$.
\end{proof}

\begin{lema}\label{x-car}
Let $K$ be a model of ACVF$_{(0,0)}$.
\begin{enumerate}
\item \label{ux-car}Let $f:U_r\to K^{\times}$ be a definable group morphism.
Then the image of $f$ falls in $U_0^-$.

\item \label{uox-car} Let $f:U_r^-\to K^{\times}$ be a definable group morphism.
Then the image of $f$ falls into $U_0^-$.

\item \label{oxx-car} Let $f:\Oval^{\times}\to K^{\times}$ be a definable group morphism.
Then the image of $f$ falls in $\Oval^{\times}$.

\item \label{uh-car} Let $f:U_r\to H(b)$ be a definable group morphism, then the image of 
$f$ falls into $U_0^-$.

\item \label{uoh-car} Let $f:U_r^-\to H(b)$ be a definable group morphism.
Then the image of $f$ falls into $U_0^-$.

\item \label{oxh-car} 
Let $f:\Oval^{\times}\to H(b)$ be a definable group morphism, then the image of
$f$ falls into $\Oval^{\times}$.

\end{enumerate}
\end{lema}

\begin{proof}
For \ref{ux-car} consider the composition $g:U_r\to K^{\times}\to \Gamma$.
If this morphism is not trivial,
then because $U_r$ has no proper definable subgroups of finite index
the image is infinite. The kernel of $g$ is a definable
proper subgroup of $U_r$ so it is contained in $U_r^-$.
We then get surjective definable group morphism
 $\text{Im}(g)=U_r/\text{ker}(g)\to U_r/U_r^-\cong k$. As
$k$ and $\Gamma$ are orthogonal any group morphism $\text{Im}(g)\to k$ 
has finite image, see Lemma \ref{val-and-res}, and this is a contradiction.
We conclude that the image of $f$ falls in $\Oval^{\times}$.

Now consider then composition $U_r\to \Oval^{\times}\to k^{\times}$.
The definable subgroups of $k^{\times}$ are $k^{\times}$ and finite.
As $U_r$ has no proper definable subgroup of finite index, if the map
is not trivial it is surjective.
As before we obtain a definable surjective group morphism $k^{\times}\to k$.
We conclude, by Lemma \ref{def-to-alg-acf} and \ref{alg-group-map},
that the map is trivial, in other words that the image of 
$f$ falls into $U_0^-$.

For \ref{uox-car} note that $U_r^-=\bigcup_{t>r}U_t$.

For \ref{oxx-car}, we already have that $f(U_0^-)\subset U_0^-$, so the map
$f$ produces a map $k^{\times}\to \Gamma$. As $k$ and $\Gamma$ are orthogonal
this map is trivial, as required.

The items \ref{uh-car}, \ref{uoh-car} and \ref{oxh-car} 
have similar proofs, replacing $\Gamma$ by $C(v(b))$ where
required.
\end{proof}

\begin{lema}\label{u-mor}
If $f:U_r\to U_0^-$ is a definable group morphism then 
$f$ is of the form $x^q$ for a rational number $q$.

The definable group morphisms $f:U_r^-\to U_0^-$ are of the same form.
\end{lema}

\begin{proof}
Let $f:U_r\to U_0^-$ be a definable group morphism.
We have that there exist integers $n> 0$ and $m$ such that
the maps $f(x^n)=x^m$ for every $x\in U_{\frac{1}{n}r}$.
This is by Lemmas \ref{def-to-alg-mor} and \ref{alg-group-map}. This finishes the proof.

The proof for $U_r^{-}$ is identical.
\end{proof}

\begin{lema}\label{ox-mor} Let $K$ be a model of ACVF$_{(0,0)}$.
If $f:\Oval^{\times}\to \Oval^{\times}$ is a definable group morphism, then there is
an integer $m$ such that $f(x)=x^m$.
\end{lema}

\begin{proof}
We know there exist co-prime integers $m$ and $n$ such that
$f(x)^n=x^m$ for $x\in U_0^-$, by Lemma \ref{u-mor}.
The map $\Oval^{\times}\to \Oval^{\times}$, $x\mapsto f(x)^nx^{-m}$ is a definable
group map that is trivial in $U_0^-$ and so factors as a group map
$k^{\times}\to U_0^{\times}$. Such a map has finite image by Lemma \ref{res-to-home}, 
and so because 
$k^{\times}$ has no proper subgroups of finite index it is trivial.
In other words $f(x)^n=x^m$ for every $x\in \Oval^{\times}$.
Take $x$ an arbitrary point in $\Oval^{\times}$, and
$y$ and $y'=y\alpha$ two different
$n$-roots of $x$, where $\alpha$ is a primitive $n$-root of unity. 
Then $y^m=y'^m=f(x)$, so $\alpha^m=1$. This can only happen if $n=1$ as required. 
\end{proof}

\begin{lema}\label{xx-mor}
Let $K$ be a model of ACVF$_{(0,0)}$.
Suppose $f:K^{\times}\to K^{\times}$ is a definable group map.
Then it is of the form $x\mapsto x^n$ for an integer $n$.
\end{lema}

\begin{proof}
We know there exists an integer $n$ such that $f(x)=x^n$ for every $x\in \Oval^{\times}$,
by Lemma \ref{ox-mor}. Then 
the map $K^{\times}\to K^{\times}$ given by $x\mapsto f(x)x^{-n}$ is a definable
group morphism, and its kernel contains $\Oval^{\times}$.
So it factors as a group map $\Gamma\to K^{\times}$, which we know must be 
trivial, by Lemma \ref{val-to-home}, as required.
\end{proof}

Recall that for $c\in \Gamma$ we set $C(c)=O(c)/\mathbb{Z}c$.

\begin{lema}\label{c-mor}
Let $\Gamma$ be a divisible ordered abelian group.
Then there are no nontrivial definable group morphisms $C(c)\to \Gamma$
or $\Gamma \to C(c)$.
If $f:C(c)\to C(b)$ is a nontrivial definable group morphism
then $o(c)=o(b)$, thus $C(c)=C(b)$, and there is a non-zero $q\in \mathbb{Q}$ such that $f$ is induced by the map $x\mapsto qx$.
\end{lema}

\begin{proof}
Suppose we have a type-definable morphism $f:o(c)\to \Gamma$, say
type-definable over a parameter set $A$.
Consider the set $X$ of elements $x\in o(c)$ which are greater than any $A$-definable element. These conform a complete type over $A$, so there is a rational 
number $q$ and an $A$-definable element $a$ such that
$f(x)=qx+a$ for all $x$ in $X$. As any point of $o(c)$ is a difference of two
elements of $X$ we have that $f(x)=qx$ for every $x\in o(c)$.

This implies that any ind-definable morphism $f:O(c)\to \Gamma$ is of the form
$x\mapsto qx$, see the uniqueness in Lemma \ref{r-extension}. If it is not trivial then it is injective, so there is no
nontrivial definable morphism $C(c)\to \Gamma$ as otherwise the composition
$O(c)\to C(c)\to \Gamma$ is nontrivial and not injective.

A similar proof shows
 that every type-definable group morphism $o(c)\to C(c)$ is a multiplication 
by a rational number followed by the projection map $O(c)\to C(c)$.
We conclude that every ind-definable group morphism $O(c)\to C(c)$
lifts to an ind-definable group morphism $O(c)\to O(c)$ and any ind-definable
morphism $O(c)\to O(c)$ is of the form $x\mapsto qx$.

Now we prove that there is no nontrivial 
type-definable group morphism $f:o(b)\to C(c)$
if $c\in o(b)$. Indeed, suppose $f$ is defined over $A$, and $A$ includes $d$
and $c$. Identifying $C(c)$ with the interval $[0,c)$,
we have as before that, for $X$ the set of elements of 
$o(b)$ which are larger than every $A$-definable element in $o(b)$, 
then $f(x)=qx+a$ for some rational number $q$, a constant
$a$ which is $A$-definable, and every $x\in X$.
But the restriction $f(x)\in [0,c)$ is 
impossible unless $q=0$, so we conclude $f$ is constant. As every element
in $o(b)$ is a difference of two elements in $X$ we conclude that 
$f$ is trivial.
This implies that there is no non-trivial group morphism $\Gamma\to C(c)$
and that if $C(c)\to C(b)$ is a nontrivial group morphism then 
$o(c)\subset o(b)$.

Now suppose that $o(c)\subset o(b)$ and that 
$f:o(c)\to C(b)$ is a type-definable
group morphism. As before one obtains that $f(x)=\overline{qx}$ for a rational
number.
Now if $O(c)\to C(b)$ is a nontrivial ind-definable 
group morphism, then the map lifts
to an ind-definable group map $O(c)\to O(b)$ given by multiplication by a 
rational number. This map factors as 
$C(c)\to C(b)$ only in the case described in the hypothesis.
\end{proof}

\begin{lema}\label{h-mor}
Let $K$ be a model of ACVF$_{(0,0)}$.
There are no nontrivial group maps $H(c)\to K^{\times}$ or $K^{\times}\to H(c)$.
If $H(c)\to H(b)$ is a nontrivial group map then there are integers
such that $b^m=c^n$ and the map comes from $x\mapsto x^n$.
\end{lema}

\begin{proof}
We know that for a definable group map $H(c)\to K^{\times}$ the image
of $\Oval^{\times}$ falls into $\Oval^{\times}$, see Lemma \ref{x-car}, so we obtain a group map
$C(v(c))\to \Gamma$, which we know must be trivial, by Lemma \ref{c-mor}. In other words
the image of $H(c)$ falls into $\Oval^{\times}$.
The map $f:H(c)\to \Oval^{\times}$ produces a surjective 
map $k^{\times}\to C(v(c))/A$
where $A$ is the image of the kernel of $H(c)\to \Oval^{\times}\to k^{\times}$ 
in $C(v(c))$.
If $A$ is finite we have a surjective map $k^{\times}\to C(\frac{1}{n}v(c))$,
which is a contradiction with Lemma \ref{val-and-res}.
If it is infinite then the kernel of $H(c)\to k^{\times}$ contains
$\Oval^{\times}$, by Lemmas \ref{v-opa} and \ref{opa-group}, 
so that the image of $\Oval^{\times}$ under $H(c)\to \Oval^{\times}$
falls into $U_0^-$. By the characterization of such maps in Lemma \ref{ox-mor}, we conclude
that $f$ is trivial when restricted to $\Oval^{\times}$.
Then the map $f$ factors as a group map $C(v(c))\to \Oval^{\times}$, which 
can only be trivial, see Lemma \ref{val-to-home}.

Now suppose than $K^{\times}\to H(c)$ is a definable group map.
We know that the image of $\Oval^{\times}$ falls into $\Oval^{\times}$, see Lemma \ref{x-car},
 so this map
produces a group map $\Gamma\to C(v(c))$ which must be trivial by Lemma \ref{c-mor}.
In other words the image of $K^{\times}$ falls into $\Oval^{\times}$.
Such a map is trivial, see Lemma \ref{xx-mor}.

Now suppose $f:H(c)\to H(b)$ is a definable group map.
Then, by Lemma \ref{x-car}, it produces a map $C(v(c))\to C(v(b))$. Assume first this map
is trivial. Then the image of $f$ falls into $\Oval^{\times}$, and we have proven
$f$ must be trivial.
We have then that the map $C(v(c))\to C(v(b))$ is not trivial.
In this case we have that $O(c)=O(b)$, and also that $f(o(c))$ 
falls into $o(b)$, see Lemma \ref{x-car}. So the restriction of $f$ to $o(c)$ lifts to a type-definable group map 
$h:o(c)\to O(b)$. The restriction of this group map to $\Oval^{\times}$
is $x\mapsto x^n$ for some integer $n$, by Lemma \ref{ox-mor}. 
Now the map $g:o(c)\to O(b)$
given by $x\mapsto f(x)x^{-n}$ is a type-definable group map which factors
as $o(v(c))\to K^{\times}$. This map is trivial by Lemma \ref{val-to-home}.
We conclude that $g:O(c)\to O(b)$ given by $x\mapsto x^n$, equals
$f$ when composed by the canonical $O(b)\to C(b)$ and restricted to $o(c)$.
So if we denote $p$ the canonical projections, 
then we have that $pg=fp$, by the uniqueness part of Lemma \ref{r-extension}. 
We conclude that $c$, $b$ and $f$ are as described.
\end{proof}

We are now left with the elliptic curve case.

\begin{lema}\label{e-car}
Suppose $E$ and $F$ are elliptic curves.

\begin{enumerate}
\item \label{er-car}
Suppose $f:E_r\to F$ is a non-trivial definable group morphism with $r>0$. 
Then it is injective and its image is of the form $F_s$.

\item \label{ero-car}
Suppose $f:E_r^-\to F$ is a non-trivial definable group morphism. 
Then it is injective and its image is of the form $F_s^-$.
\item \label{e0-car}
Suppose $f:E_0\to F$ is a non-trivial definable group morphism. 
Then the image of $f$ in $F$ is $F_0$.
\item \label{eo0-car}
Suppose $f:E_0^-\to F$ is a non-trivial definable group morphism. Then 
$f$ is an isomorphism onto $F_0^-$.
\item \label{erhe-car}
Suppose $f:E_r\to H_F(b)$ is a definable group morphism with $r>0$. 
Then the image of $f$ in $F$ falls into $F_0^-$.
\item \label{erohe-car}
Suppose $f:E_r^-\to H_F(b)$ is a definable group morphism 
Then the image of $f$ in $F$ falls into $F_0^-$.
\item \label{e0he-car}
Suppose $f:E_0\to H_F(b)$ is a non-trivial definable group morphism.
Then the image of $f$ in $F$ is $F_0$.
\end{enumerate}
\end{lema}

\begin{proof}
For item \ref{er-car} we first show that $f(E_r)\subset F_0^-$. This is proven in the same
way as Lemma \ref{x-car}. 

Now we show that $f$ is injective.
From Lemma \ref{def-to-alg-mor} and Lemma \ref{alg-group-map} we have that there is an
integer $n$ and a group morphism $g:E\to F$ such that 
$f(nx)=g(x)$ for every $x\in E_{\frac{1}{n}r}$. As $g$ has finite kernel, $f$ as finite kernel. But $E_0^-$ is torsion free so it is injective.

Now we show that the image of $E_r$ is of the form $F_s$. By Proposition \ref{acvf-ell}
we have that the image of $E_r$ is of the form $F_s$ or of the form $F_s^-$.
As $E_r$ has a maximal proper definable subgroup, we conclude that the second case
does not occur.

Item \ref{ero-car} is similar and omitted.

For item \ref{e0-car}, showing that $f(E_0)\subset F_0$ is similar to the proof of Lemma
\ref{x-car}. To show that it is surjective, note that there exists an integer $n$ and a
group morphism $g:E\to F$ such that $f(nx)=g(x)$ for $x\in E_0$.
We also have a group morphism $h:F\to E$ such that $gh(x)=mx$ for an integer $m$.
By what we have just proven $h$ and $g$ restrict to $F_0\to E_0$ and $E_0\to F_0$.
As $F_0$ is divisible, we conclude that $f(E_0)=F_0$.

For item \ref{eo0-car} we have already seen that $f$ is injective and falls into $F_0^-$, 
and surjectivity is similar to the case \ref{e0-car}.

The items \ref{erhe-car}, \ref{erohe-car}, and \ref{e0he-car} are similar to 
\ref{er-car}, \ref{ero-car} and \ref{e0-car} respectively.
\end{proof}

\begin{lema}\label{isog-to-iso}
Suppose $K$ is a model of ACVF$_{(0,0)}$.
Suppose $E$ and $F$ are isogenous elliptic curves. Then $E$ and $F$ have the
same reduction type. Also $E$ is a nonstandard Tate curve if and only if 
$F$ is. If $E$ and $F$ are of multiplicative type then $E_0\cong F_0$.
In any case $E_0^-\cong F_0^-$.
\end{lema}

\begin{proof}
Let $f:E\to F$ be a non-trivial definable group morphism.

By Lemma \ref{e-car} we have that
$f(E_0)=F_0$ and $f(E_0^-)=F_0^-$. Suppose $E$ is multiplicative and $F$ is of good type.
Then we get a surjective group morphism
$E_0/E_0^-=k^{\times}\to F_0/F_0^-=\tilde{F}(k)$, 
which by Lemma \ref{alg-group-map} can not happen. 

Suppose $E$ and $F$ are of multiplicative type. Then the kernel of $f:E_0\to F_0$ is finite
cyclic and so it factors through the multiplication by $n$ map $t_n:E_0\to E_0$ into
an isomorphism $E_0\to F_0$.
\end{proof}

\begin{lema}\label{goo-ell-mor}
Suppose $K$ is a model of ACVF$_{(0,0)}$.
Suppose $E$ and $F$ are elliptic curves over $K$. Suppose
$f:E\to F$ is a definable group morphism.
Then $f$ is an algebraic group morphism.

In particular, 
if $E$ and $F$ are isomorphic as definable groups, then they are isomorphic
as algebraic groups.
\end{lema}

\begin{proof}
Suppose $f:E\to F$ is a definable group morphism.
By Lemma \ref{def-to-alg-mor} and Lemma \ref{alg-group-map} we conclude
there is an elliptic curve $E'$ and isogenies $p:E'\to E$ and $q:E'\to F$
such that $fp=q$. As $p$ is surjective we conclude $f$ is definable
in ACF$_0$, and so it is an algebraic group morphism by Lemma \ref{def-to-alg-acf}.
\end{proof}

\begin{lema}\label{he-mor}
Suppose $K$ is a model of ACVF$_{(0,0)}$ and $E$ and $F$ are elliptic curves of 
multiplicative type.
Suppose $f:E\to F$ is an isogeny. Then $f$ restricts to a surjection $o_E\to o_F$ and this 
restriction extends in a unique way to an ind-definable map $\tilde{f}:O_E\to O_F$.
If we denote $p$ the canonical projections $O_E\to E$ and $O_F\to F$ then we have
$p\tilde{f}=fp$.

Now suppose $g:H_E(a)\to H_F(b)$ is a definable group map. 
Then $g$ restricts to a surjective group map $o_E(a)\to o_F(b)$ which extends in a unique
way to an ind-definable morphism $\tilde{g}:O_E(a)\to O_F(b)$. If we denote $p$
the canonical projections $O_E(a)\to H_E(a)$ and $O_F(b)\to H_E(b)$ then $p\tilde{g}=gp$.
Further, there exists an isogeny
$f:E\to F$ and a non-zero integer $n$ such that $\tilde{g}t_n=\tilde{f}$, for $t_n$
the multiplication by $n$ map $O_E(a)\to O_E(a)$.

Reciprocally if $c$ is an element of $O_E$ with $nc=a$, 
then any isogeny $f:E\to F$ such that $f$ is $0$ at the $n$-torsion of $E_0$
and $\tilde{f}(c)\in \mathbb{Z}b$ produces a  morphism $g:H_E(a)\to H_F(b)$
determined by the equations $\tilde{g}t_n=f$ and $p\tilde{g}=gp$.
\end{lema}

\begin{proof}
Suppose $f:E\to F$ is an isogeny. By Lemma \ref{e-car} we have that
$f(E_0)=F_0$, so $f$ induces a definable group map 
$\bar{f}:C(d_E)\to C(d_F)$ 
where $d_E$ and $d_F$ are the valuations of the discriminants of 
minimal Weierstrass equations of $E$ and $F$ respectively.
We know such a map satisfies $\bar{f}(o(d_E))=o(d_F)$, see Lemma \ref{c-mor}. 
We conclude that
$f(o_E)=o_F$, and so $f$ extends in a unique way to an ind-definable morphism
$\tilde{f}:O_E\to O_F$, by Lemma \ref{r-extension}. That $p\tilde{f}=fp$ follows
from the uniqueness in that lemma.

Now suppose $g:H_E(a)\to H_F(b)$ is a definable group morphism. As before we have that
$g(E_0)=F_0$, $g(o_E(a))=o_F(b)$ and
existence of $\tilde{g}$, with $p\tilde{g}=pg$. 
By Lemmas \ref{def-to-alg-mor} and \ref{alg-group-map} we conclude that there exists
$n$ and $f:E\to F$ with $g(nx)=f(x)$ for $x\in E_0$.
Now $f-gt_n$ is a morphism $o_E(a)\to o_F$ which factors as $o(v(a))\to o_F$.
By Lemma \ref{val-to-home} we conclude $f=gt_n$ in $o_E(a)$.
By the uniqueness in Lemma \ref{r-extension} we conclude 
$\tilde{f}=\tilde{g}t_n$, as required.
\end{proof}

\begin{prop}\label{acvf-iso}
Let $K$ be a model of ACVF$_{(0,0)}$. Let $G$ and $H$ be two basic groups.
\begin{enumerate}
\item\label{k-iso} If $G=K$ then $G$ is isomorphic to $H$ if and only if 
$H=K$.
\item\label{o-iso} If $G=\Oval$ then $G$ is isomorphic to $H$ if and only if $H=\Oval$.
\item\label{m-iso} If $G=\Mval$ then $G$ is isomorphic to $H$ if and only if $H=\Mval$.
\item\label{x-iso} If $G=K^{\times}$ then $G$ is isomorphic to $H$ if and only if 
$H=K^{\times}$
\item\label{ox-iso} If $G=\Oval^{\times}$ then $G$ is isomorphic to $H$ if and only if 
$H=\Oval^{\times}$.
\item\label{uo-iso} If $G=U_r^-$ then $G$ is isomorphic to $H$ if and only if 
$H=U_{qr}^-$ for $q$ a positive rational number.
\item\label{u-iso} If $G=U_r$ then $G$ is isomorphic to $H$ if and only if 
$H=U_{qr}$ for $q$ a positive rational number.
\item\label{h-iso} If $G=H(a)$ then $G$ is isomorphic to $H$ if and only if 
$H=H(a)$.
\item\label{e0-iso} If $G=E_0$ for $E$ an elliptic curve of multiplicative type 
then $G$ is isomorphic to $H$ if and only if 
$H=F_0$ for $F$ an elliptic curve isogenous to $E$, necessarily of the 
same reduction type.
\item \label{eo0-iso} If $G=E_0^-$ for $E$ an elliptic curve, then $G$ is 
isomorphic to $H$ if and only if $H=F_0^-$ for an elliptic curve $F$
isogenous to $E$ necessarily of the same reduction type as $E$.
\item\label{e-iso} If $G=E$ for $E$ an elliptic curve of good type 
then $G$ is isomorphic to $H$ if and only if 
$H=F$ for $F$ an elliptic curve isomorphic to $E$ as elliptic curves.
\item\label{er-iso} If $G=E_r$ with $r>0$, and $E$ an elliptic curve 
then $G$ is isomorphic to $H$ if and only if 
$H=F_s$ for an elliptic curve $F$, necessarily of the same reduction type as $E$,
and such that there is an isogeny $f:E\to F$ and a positive integer $n$ such that
$wf(x)=s$ for an $x$ with $w(x)=\frac{1}{n}r$ 
\item\label{ero-iso} If $G=E_r^-$ with $r>0$, and $E$ an elliptic curve 
then $G$ is isomorphic to $H$ if and only if 
$H=F_b^-$ for an elliptic curve $F$, necessarily of the same reduction type as $E$,
and such that there is an isogeny $f:E\to F$ and a positive integer $n$ such that
$wf(x)=b$ for an $x$ with $w(x)=\frac{1}{n}a$.
\item\label{he-iso} If $G=H_E(a)$ for an elliptic curve $E$ of multiplicative type,
then $G$ is isomorphic to $H$ if and only if $H=H_F(b)$ for $F$ such that there
is an isogeny $f:E\to F$ which restricts to a map $f:o_E\to o_F$, such that if 
$\tilde{f}:O_E\to O_F$ is the unique ind-definable extension of $f|o_E$, then
$\tilde{f}(c)=b$ for $c$ such that $nc=a$ where $n$ is the number of elements in the 
kernel of $f|E_0$. Any isogeny $f:E\to F$ restricts to $f:o_E\to o_F$ so if the kernel
of $f|E_0$ has size $n$, then obtain an isomorphism $H_E(nc)\cong H_F(\tilde{f}(c))$
\end{enumerate}
\end{prop}

\begin{proof}
For \ref{k-iso}, \ref{o-iso} and \ref{m-iso}, note that if $H$ is isomorphic to $G$
then it must be basic of additive type by Lemma \ref{def-to-alg-mor} and 
Lemma \ref{alg-group-map}. The result now follows from Lemma \ref{add-mor}.

For \ref{x-iso} note that by Lemma \ref{def-to-alg-mor} and Lemma \ref{alg-group-map}
we have that $H$ must be of multiplicative type. Also, by Lemma \ref{x-car},
$H$ can not be of type $\Oval^{\times}$, $U_r^-$ or $U_r$. Finally $H$ can not be of type
$H(a)$ either by Lemma \ref{h-mor}

For \ref{ox-iso} as before $H$ has to be of multiplicative type and by Lemma \ref{x-car}
it can only be of the form $\Oval^{\times}$.

For \ref{uo-iso} as before $H$ has to be of multiplicative type and by Lemma \ref{x-car}
it can not be of the form $H(a)$. We conclude by Lemma \ref{u-mor}

Item \ref{u-iso} is the same as \ref{uo-iso}.

For \ref{h-iso} we have already seen $H$ must be of the form $H(b)$, and by 
Lemma \ref{h-mor} we have $b=a$.

For \ref{e0-iso} and \ref{eo0-iso} see Lemma \ref{isog-to-iso} and 
Lemmas \ref{def-to-alg-mor}, \ref{alg-group-map} and \ref{e-car}.

For \ref{er-iso} and \ref{ero-iso} see 
Lemmas \ref{def-to-alg-mor}, \ref{alg-group-map} and \ref{e-car}.

For \ref{e-iso} see Lemma \ref{goo-ell-mor}.

For \ref{he-iso} see 
Lemmas \ref{def-to-alg-mor}, \ref{alg-group-map}, \ref{e-car}, and \ref{he-mor}.
\end{proof}
\bibliographystyle{plain}
\bibliography{one-dimensional}
\end{document}